\documentclass[11pt]{article}
\usepackage{mylatexsetting}

\usepackage{pgfplots}
\pgfplotsset{compat=1.17}
\title{On Preserving or Reversing Higher-Order Unimodality and Convexity by Sign-Regular Kernels}
\author{Zakaria Derbazi \\{\it\small Queen Mary University of London}}
\newcommand{\negL}[1]{\overline{\mathcal{#1}}}
\newcommand{\eL}{\mathcal{L}}

\definecolor{darkblue}{rgb}{0.0, 0.0, 0.55}

\def\myabstract{
	This work investigates preserving and reversing unimodality and convexity properties for sequences under transformations defined by sign-regular kernels. It is shown that these transformations only preserve these properties if the kernels are totally positive of order three or their additive inverse is totally negative of order three. In contrast, these transformations reverse these properties if the underlying kernel is totally negative or if its additive inverse is a totally positive kernel, both of order three. Furthermore, these results are extended to higher-order convex and multimodal sequences. These findings, which expand upon Karlin's earlier results on convexity, form the basis for deriving sufficient conditions for the preservation or reversal of higher-order convexity or generalised unimodality of a quotient of sequences, where both the numerator and denominator are transformations by the same sign-regular kernel.
}

\usetikzlibrary{arrows.meta,automata, chains, positioning, arrows}
\tikzset{%
	>={Latex[width=2mm,length=2mm]},
	base/.style = {rectangle, rounded corners, draw=black,
		minimum width=2cm, minimum height=1cm,
		text centered, font=\sffamily},
	activityStarts/.style = {base},
	startstop/.style = {base, fill=green!5},
	startstop2/.style = {base, fill=purple!5},	
	activityRuns/.style = {base, fill=green!30},
	activityRuns2/.style = {base, fill=purple!40},	
	finalOutput/.style = {base, fill=black!30},	
	process/.style = {base, minimum width=2.5cm, fill=orange!15,
		font=\ttfamily},
	process2/.style = {base, minimum width=2.5cm, fill=orange!50,
		font=\ttfamily},		
}
\begin{document}
	\maketitle
	\begin{abstract}
		\myabstract
	\end{abstract}

	\section{Introduction}\label{sec1} 
	
	A real matrix, finite or infinite, is said to be \textit{sign-regular} of order $r$, or  $\SR{r}$ for short if the sign of every minor of order $k=1,2,\ldots, r$ depends only on $k$. Specifically, for each $k$, there exists a number $\varepsilon_k\in \{-1,1\}$ such that  $\varepsilon_k d_k \ge 0$, where $d_k$ represent any $k$th order determinant of the matrix. We refer to $\varepsilon_k$ as the `minor sign' of order $k$. 
	A stronger form of sign regularity is when all the minors at every order $1,2,\ldots, r$ are nonnegative (resp. nonpositive). In these cases, the matrix is called \textit{totally positive} (resp. \textit{totally negative}) of order $r$,  or $\TP{r}$ (resp. $\TN{r}$) for short.
	
	These definitions extend to  kernels $K: I\times J \to \mr$, with $I, J \subset \mr$, by considering  the matrix $\left(K(x_i, y_j)\right)$, where ${i,j \in [n] \coloneqq \{1,\ldots,n\}}$, for any selection of ordered tupless $(x_1, \ldots, x_n)$ in $I$ and $(y_1, \ldots, y_n)$ in $J$. The kernel is $\TP{r}$ (resp. $\TN{r}$, or $\SR{r}$) if the corresponding matrix has that property for all such selections of ordered tuples. If the subscript $r$ is omitted in any of the previous definitions, the property  in question is understood to hold for all orders.  For example, a $\TP{}$ matrix has nonnegative minors of all orders. For sequences $(a_k)_{k \in I}$, where $I \subseteq \mn$,  the definitions apply to its kernel representation given by the matrix $A=(a_{i-j})_{i,j \in I}$, where $a_k = 0$ when $k < 0$. A sequence is $\TP{r}$ $\TN{r}$ or  $\SR{r}$ if $A$ is $\TP{r}$ $\TN{r}$ or  $\SR{r}$  \cite{KarlinTPBook}.
	
	Sign-regular kernels possess the \textit{variation-diminishing} (VD) property. This property ensures that for a kernel $\eL$ applied to a real-valued function $f$, the number of sign changes in $\eL f$ (see Definition \ref{def.sign.chg}) does not exceed the number of sign changes in $f$. Karlin demonstrated that when  these numbers are equal and strictly less than the kernel's order of sign-regularity, the functions share identical or opposite sign patterns. While this has been extended to unimodal and higher-order convex/concave functions under total positivity, the impact of sign reversals by these kernels on unimodal or convex/concave sequences remains unexplored.  Furthermore, the role and  types of sign-regularity (beyond total positivity) in preserving or reversing unimodality and convexity have not been thoroughly investigated.  
	
	In light of these limitations, this work aims to address these gaps and broaden the scope to higher-order convexity/concavity and generalised unimodality.  Motivated by applications of unimodality-preserving kernels in optimal stopping theory \cite{ZD0}, and recent work on unimodality preservation in ratios of series and integral transforms  \cite{Karp}, we investigate these properties.  However, prior studies have largely concentrated on sign-regular unimodality-preserving kernels, neglecting the reversion property.  Furthermore, the preservation property considered in recent work was incorrectly assumed to hold unconditionally for sign-regular kernels.  To rectify this, we explore the reversion property of sign-regular kernels and extend the preservation and reversal of higher-order convexity/concavity to generalised unimodality, which we refer to as $m$-modality or generalised unimodality of order $m$.  We also provide corrections and adaptations to recent findings, extending them to the multimodal setting.  Our analysis is confined to sequences, with the treatment of transformation mapping functions to sequences/functions reserved for a subsequent paper.
	
	The structure of this work is as follows: Section 2 introduces $m$-modal sequences and establishes key properties, including necessary and sufficient conditions for $m$-modality.  We also demonstrate a simplification that arises when considering a multimodal sequence as a sequence of unimodal subsequences sharing a common mode. Section 3 establishes necessary and sufficient conditions for third-order sign-regular kernels to be unimodality-preserving (resp. unimodality-reversing).  Here, unimodality reversal is defined as transforming a unimodal sequence into a bimodal or nonmodal sequence as it shall be revealed in the sequel.  This section further explores the effects of unimodality-preserving (resp. reversing) transformations and their inverses on unimodal quotients, differences, and sums.
	
	The final section turns to sign-regular kernels that preserve or reverse generalised unimodality and higher-order convexity.  Mirroring the base case of unimodality and convexity, it is shown that if a transformation defined by a sign-regular kernel is $m$-modality-preserving, then the kernel must be either totally positive or have a totally negative additive inverse, both of order $2m+1$. Conversely, for $m$-modality reversal, the kernel must be totally negative, or its additive inverse totally positive, both of order $2m+1$.  For $m$-convex sequences, the kernel requirements are identical to the $m$-modality case, but the necessary order of sign-regularity is only $m+1$, and stronger conclusions are obtained. These results extend Karlin's earlier work on  $m$-convexity-preserving transformations.  This section also presents sufficient conditions for the $m$-modality and $m$-convexity (resp. $m$-concavity) of differences and quotients of transformations of sequences whose quotient is  $m$-modal (resp. $m$-convex) --results that may be of independent interest.  For background on classical monotonicity criteria for quotients of functions, see, for example,  \cite{mao, yang, pinelis} and the references therein, and  \cite{Karp} for the unimodality criteria.   An application to domination polynomials and illustrative numerical examples are provided.

	\subsection*{Notation and Conventions}
	Throughout this paper, $\mn$ and $\mr$ denote, respectively, the set of positive integers, and the set of real numbers;  $[n]$ is the set of the first $n$ positive integers. We drop the subscript from the notation of a sequence and write $u$ or $(u_k)$ to mean $(u_k)_{k \in I}$, where $I \subseteq\mn$ if no ambiguity arises. For two totally ordered sets $A$ and $B$, we write $A \prec B$ to mean that $a < b$ holds for all $a \in A$ and all $b \in B$. Operations on sequences are understood to be term by term. That is, for a sequence $u$ and scalar $\lambda$, the sequence $u-\lambda$ is defined as $(u_1 - \lambda, u_2 - \lambda, \cdots)$. For a given operator $\eL$, we employ the notation $\eL u$, $\eL(u)$ interchangeably. In the domain of the definition of a kernel $\mathcal{L}: I \times J \to \mr$, $I$ and $J$ are either intervals of the real line or a countable set of discrete values along the real line. In this work,  both $I$ and $J$ are countable sets, and $\mathcal{L}$ is considered a (finite or infinite) matrix. The generic terms kernel and transformation are used interchangeably.
	
	\section{Generalised Unimodality of Sequences}
	We shall first define some concepts used throughout this paper .
	\begin{define}[Number of sign changes]\label{def.sign.chg}
		The number of sign changes in a finite or infinite sequence $u \coloneqq (u_1, u_2, \ldots,)$ is given by the function
		\begin{equation*}
			S(u)  \coloneqq  S(u^\prime_{1},u^\prime_{2}, \ldots) = \sum_{i \ge 1} \abs{\sgn(u^\prime_{i+1})-\sgn(u^\prime_i)},\qquad S(0, 0, \ldots, 0) =0,
		\end{equation*}	
		where $u^\prime_1, u^\prime_{2}, \ldots$  are the relabelled  nonzero terms of $u$. 
		In the case of a real function $f$ defined on $I \subseteq \mr$, its number of sign changes on $I$ is given by 
		\begin{equation*}
			S(f) = \sup_{ \substack{x_1, \ldots, x_n \in I \\[2pt] x_1 < x_2 < \ldots < x_n}} S[f(x_1), f(x_2), \ldots, f(x_n)],
		\end{equation*}
		where  $n$ is taken to be arbitrary but finite \cite{SchoenbergVD1}. We further define the `maximum number of sign changes' in a sequence (or a function) $u$ resulting from a shift by a constant value 
		\begin{equation}\label{def.s+}
			S^+(u) \coloneqq \max_{\lambda \in \mr} S(u-\lambda).
		\end{equation}

	\end{define}
	
	In the rest of the definitions, $m$ is a positive integer.
	
	\begin{define}[$m$-modal function]\label{def.multimodel.function}
		A real function $f : I \to \mr$ is $m$-modal on $I \subseteq \mr$, if $I$ can be partitioned into $2m-1$ or $2m$ disjoint intervals  $I_1 \prec J_1 \prec I_2 \prec J_2 \prec \cdots \prec J_{m-1} \prec I_m \prec J_m$, such that
		\begin{equation}\label{def.condition.multimodel.functions1}\tag{C1}
			\begin{dcases}
				f \text{ is nodecreasing on}  &I_{k},\\			
				f \text{ is nonincreasing on } &J_{k} ,\\
			\end{dcases}
		\end{equation}
		and for  $k \in [m-1]$, there exists $x_k \in I_k$,  $y_k \in J_k$ and $x_{k+1} \in I_{k+1}$ such that
		\begin{equation}\label{def.condition.multimodel.functions2}\tag{C2}
			f(x_k) > f(y_k) <  f(x_{k+1}),
		\end{equation}	
	\end{define}
	
	A classic example of a bimodal  (2-modal) function is  the $U$-shaped symmetric beta probability density function $f(x) = {\left[x(1-x)\right]^{\alpha-1}}/{{\rm B}(\alpha, \alpha)}$, where  $x \in [0,1]$, $\alpha > 1$ is the shape parameter and $\rm B$ is the beta function. $f$ has a maximum value $1/2$ located at $x=0$ and $x=1$. Here, we can use the partition of interval $I$ into  $I_1=\{0\}, J_1=(0.0.5], I_2=(0.5,1)$ and $J_2=\{1\}$.

	For a real sequence $f$ ($I \subseteq \mn$), we adjust the definition of $m$-modality as follows.
	\begin{define}[$m$-modal sequence]\label{def.multimodel}
		A   real  sequence $u \coloneqq (u_1, u_2, \ldots)$  supported on the lattice of positive integers is said to be 
		\textit{$m$-modal} (or modal of order $m$), if there exists $2m-1$ disjoint (and nonempty) intervals 
		$		[a^-_{1}, a^+_{1}] \prec [b^-_{1}, b^+_{1}] \prec [a^-_{2}, a^+_{2}] \prec [b^-_{2}, b^+_{2}] \prec  \cdots  \prec [b^-_{m-1}, b^+_{m-1}] \prec  [a^-_{m}, a^+_{m}]  $
		in  $\mn$ such that  for $k \in [m]$ 	
		\begin{equation}\label{def.condition.multimodel.sequence1}\tag{D1}
			\begin{dcases}
				u_{j} \le  u_{j+1} &~ \text{all } ~ b^+_{k-1} \le  j  <  a^-_{k},\\
				u_{j} =  u_{j+1} &~ \text{all } ~  a^-_{k} \le j  < a^+_{k} \text { and } b^-_{k-1} \le j  < b^+_{k-1},\\
				u_{j} \ge  u_{j+1} &~ \text{all } ~  a^+_{k} \le j <  b^-_{k}  ,
			\end{dcases}
		\end{equation}
		and for $k \in [m]$
		\begin{equation}\label{def.condition.multimodel.sequence2}\tag{D2}
		u_{b^-_{k-1}}	< u_{a^-_{k}} > u_{b^-_{k}},
		\end{equation}			 
		with the convention  $b^-_0 =b^+_0 = 0$, $b^-_m = \infty$ and  $u_0 = u_\infty= \inf u$. In the  case where $u$ is a finite sequence of length $N$, set $b^-_m = N+1$ and  $u_{N+1}=u_0$. $[a^-_1, a^+_1], \ldots,  [a^-_m, a^+_m]$ are  the\textit{ intervals of modes} of $u$ and $(0, a^-_1], [a^+_1, b^-_1], \ldots, [b^+_{m-1}, a^-_m], [a^+_m, \infty)$  are the corresponding  \textit{monotonicity intervals}.
	\end{define}

	\begin{rem}
		$ $
		\begin{itemize}
			
			\item 
			Our definition of $m$-modality agrees, in the unimodal case, with Khinchine's definition and as a characterisation property of single-humped  probability density functions (see Chapter 4 in \cite{Bertin}). 	Some optimisation literature considers bimodal functions  (in our sense) to be unimodal as they assumed quasiconvex (see Section 3.4 in \cite{Boyd}).  This contrasts sharply with the definition of unimodality common in probability texts such \cite{Dharmadhikar} (see p. 85), which equates unimodality with quasiconcavity.

			\item By our definition, any constant sequence is unimodal (1-modal or modal of order 1), and the  infinite sequence $u$ with general term $u_k=k$  (or $u_k = 1-1/k$) is not unimodal,  as it lacks an interval of modes satisfying (\ref{def.condition.multimodel.sequence1}). However, if $u_k = 1/k$, then $u$ is unimodal.

			\item Definitions \ref{def.multimodel.function} and \ref{def.multimodel} are  equivalent. In the former, each interval of modes is contained within $I_k$, whereas in the latter definition, $I_k \coloneqq [a^-_k, a^+_k]$ itself constitutes an interval of modes.	
			\item Based on Definition \ref{def.multimodel.function}, $f$ need not be continuous on $I$.

			\item The  definitions imply that $f$ is unimodal in each interval $I_k \cup J_k$ and $u$ is unimodal on each of 
			$(0, b^{+}_1], [b^{-}_1, b^{+}_2], \ldots, , [b^{-}_{m-2}, b^{+}_{m-1}],  [b^-_m, \infty)$.
			\item Periodic functions on an open interval or infinite alternating sequences, such as $(0,1,0,1,\ldots)$, are considered infinitely modal. 
			\item Definition \ref{def.mmodal} introduces other criteria for determining the $m$-modality of a sequence.
		\end{itemize}
	\end{rem}
	
	It is natural to view $m$-modal sequences as a collection of $m$ unimodal sequences.  We formalise this intuition in the next definitions.
	\begin{define}[Unimodal $m$-partition and $m$-decomposition of a sequence]\label{def.patition}
		Consider the interval $I \subseteq \mn$ partitioned into  $m \ge 1$ (disjoint) intervals $I_1, \ldots, I_m$, such that $I_{i} \prec I_{j} $ for $1 \le i < j \le n$ and each inner interval $I_k$, where $1 < k < m$,  contains at least two points while the end intervals $I_1, I_m$ may be singletons. Let $J_1, \ldots, J_m$ be `enlarged' versions of the previous intervals   such that each $J_k$ is the interval $I_{k}$ extended to contain the last point of the preceding interval $I_{k-1}$, when $k > 1$, and the first point in the subsequent interval $I_{k+1}$ when $k < m$.
		The subsequences $ u^{(1)}\coloneqq (u_k)_{k \in I_1}, \ldots, u^{(m)}\coloneqq (u_k)_{k \in I_m}$ form an \textit{$m$-partition }of $u$, denoted by $u = (u^{(1)}, \ldots, u^{(m)})$, where $u$ is obtained by concatenating $u^{(1)}, \ldots, u^{(m)}$. Conversely, the sequences $v^{(1)}\coloneqq (u_k)_{k \in J_1}, \ldots, v^{(m)}\coloneqq (u_k)_{k \in J_m}$ form an \textit{$m$-decomposition} of $u$, which is essentially $u$ with $2m-2$ duplicated terms. An $m$-partition (resp., $m$-decomposition) is unimodal if each of its components is unimodal.
	\end{define}
	To ensure that all components in a unimodal $m$-partition (resp. $m$-decomposition) of a given sequence share the same mode, we propose the following transformation:
	\begin{define}[Max-alignment and mode-alignment]\label{def.mode-alignment}
		Let $(u^{(1)}, \ldots, u^{(m)})$ be an $m-$partition  (or $m$-decomposition) of a given real sequence $u$, where each subsequence $u^{(k)}$ has a finite maximum value, which we denote $d_k$. Let $d_{*}\coloneqq \min_k d_k$. The \textit{max-aligned} $m$-partition (resp. $m$-decomposition) of $u$,  is the sequence $u^\circ \coloneqq  (u^{(1)}-\lambda_1, \ldots, u^{(m)}-\lambda_m)$, where  $\lambda_1 = d_1 -d_*, \ldots, \lambda_m = d_m -d_*$. If the original sequence is a unimodal $m$-partition (resp. $m$-decomposition), then we call this transformation the \textit{mode-aligned} $m$-partition (resp. $m$-decomposition) of $u$.
	\end{define}

	\subsection{The sufficient and necessary condition for  $m$-modality}
	To facilitate our treatment in the rest of this paper, we introduce the following definition:
	\begin{define}[$M$-Regular sequence]\label{def.regular}
		A sequence $u$ defined on $I$ is regular or $M$-regular if there exists an $M$ such that $(u_k)_{k \ge M}$ is nonincreasing. We omit the $M$ if there is no use for it and call such sequences regular.
	\end{define}	
	
	\begin{thm}\label{thm.multimodal.maximum}
		Let $u$ be an infinite real sequence on $I \subseteq \mn$. $u$ is modal of order $m \in \mn$ iff $u$ is $M$-regular for some positive. Moreover, if such $M$ exists then $2 \le 2m \le M$.
	\end{thm}
	\begin{proof}
		The necessity is straightforward from the definition of $m$-modality. For sufficiency, assume there exists an index  $M \in I$ such that the tail of $u$, from $M$ onwards, is nonincreasing. The proof constructs an $m$-modal sequence from this tail.  The maximum value of the subsequence $(u_k)_{k \ge M}$  is clearly $u_M$. Suppose that this maximum occurs in each point of  the interval $[a^-_m, a^+_m]$, where $a^-_m=M$. A first unimodal subsequence  $u^\prime$ is formed using the interval $ [b^-_{m-1},\infty)$, where $0 < b^-_{m-1} < a^-_{m}  \le a^+_m < \infty$ and $b^-_{m-1}$  chosen so that $u$ satisfies 	(\ref{def.condition.multimodel.sequence1})-(\ref{def.condition.multimodel.sequence2}). If no terms remain, that is, $b^-_{m-1}=1$, then $m=1 $ and the proof is complete. Otherwise, the maximum of the remaining subsequence to the left of  $b^-_{m-1}$ is identified and assumed to occur on  the interval $[a^-_{m-1}, a^+_{m-1}]$ ($u$ is constant in this interval). Once more,  another unimodal subsequence $u^{\prime\prime}$  is formed based on the interval $[b^+_{m-2}, b^-_{m-1}]$, where $b^+_{m-2} \in (0,  a^+_{m-1})$ is chosen so that $u$ satisfies
		(\ref{def.condition.multimodel.sequence1})-(\ref{def.condition.multimodel.sequence2}). If $b^-_{m-2}=1$, then $m=2$ and the proof is concluded. If there are remaining terms of $u$ to the left  of  $b^-_{m-2}$, then this process repeats until all terms of $u$ are exhausted after forming the $m$th unimodal sequence, where $m > 1$, around the interval $[a^-_1, a^+_1]$, where $0 < a^-_1 < b^-_1$. Since  $2m-1$ disjoint intervals $[a^-_1, a^+_1], \ldots, [a^-_m, a^+_m]$ and $[b^-_1, b^+_1], \ldots, [b^-_{m-1}, b^+_{m-1}]$ exist such that $u$ satisfies (\ref{def.condition.multimodel.sequence1})-(\ref{def.condition.multimodel.sequence2}), then  this sequence is $m$-modal in accord with Definition \ref{def.multimodel}. If each of these intervals is a singleton, then there are  exactly $M/2$ peaks if $M$ is even and $(M-1)/2$ peaks if $M$ is odd. Thus, $2m \le M$. The lower bound is evident, completing the proof.
	\end{proof}
	\begin{cor}
		If $u$ is a finite sequence with $N$ terms, then it is a modal of order $m$, where $2 \le 2m \le N$.
	\end{cor}	
	\begin{proof}
		The proof follows from the previous theorem by considering an infinite sequence  $v \coloneqq (v_k)$, where $v_k =u_k$ for $k \le N$ and $v_k = v_N$ otherwise.
	\end{proof}
	\begin{cor}
		A regular sequence is modal of order $m < \infty$.
	\end{cor}
	\begin{proof}
		The proof is immediate by application of Threom \ref{thm.multimodal.maximum}
	\end{proof}
	
	\begin{thm}\label{thm.multimodal.maximum2}
		$u$  is modal of order $m \in [1, \infty]$ 
		iff the subsequence $(u_{k})_{k \ge n}$ possesses a finite maximum for each $n \in I$.
	\end{thm}			
	\begin{proof}
		For the necessary part, suppose that  $(u_{k})_{k \ge n}$ possesses a finite maximum for each $n \in I$. In the trivial case where the hypothesis of Theorem \ref{thm.multimodal.maximum} holds, $u$ is $m$-modal for some $1 \le m < \infty$. Otherwise, assume no fixed index $M$ exists such that the tail of $u$ from index $M$ onwards is nonincreasing. Split $u$ into two subsequences, $u^{-}_{1}$, $u^{+}_{1}$ such that the left subsequence $u^{-}_{1}$ is unimodal. This is possible in view of the previous corollary. Repeat the process  $k$ times, generating subsequences  $u^{-}_{k}$, $u^{+}_{k}$. at each step, where $u^{-}_{k}$ is unimodal. Since $u$ now has $k$ modes, letting $k$ approach infinity yields countable modes, which is the desired result.
		
		For the sufficiency, assume $u$ is $m$-modal for some $m \in \mn \cup \{\infty\}$ and that $(u_{k})_{k \ge n}$ has a finite maximum for each $n \in I$. First, consider $m < \infty$. There exists  $M \ge 2m$ such that $u_k \ge u_{k+1}$ for values of $k \ge M$. Thus,  $\max_{k > M} (u_{k}) = u_k$. Also $\max_{k \le M} (u_{k}) = d^*$, where $d^*$ is the maximum of the $m$ modes of $(u_{k})_{k \le M}$. Next, for $m=\infty$,   $u$ has a countable number of modes $d_1, d_2, \ldots$, and each is a finite real number. Hence, $d^{(n)} \coloneqq \max_{k \ge n} d_k$ exists and is finite for each $n \in I$. Consequently, $ \max_{k \ge n} u_k = d^{(n)}$ for each $n \in I$,as required.
	\end{proof}
	\begin{cor}
		Every infinite sequence is either modal of order $m \in \mn$, infinitely modal, or nonmodal.
	\end{cor}
	\begin{proof}
		The proof directly applies the previous theorem. If the tail of an infinite sequence has  a finite maximum, then the sequence is $m$-modal for some $m \in \mn \cup \{\infty\}$. If the tail lacks a finite maximum, the sequence has no mode and is, therefore, nonmodal (or zeromodal).
	\end{proof}
	
	While the preceding results are important, they do not directly reveal the number of modes, which remains unknown.  Therefore, we now analyse the properties of the unimodal $m$-decomposition of an $m$-modal sequence.
	\begin{assertion}\label{lem.numsign.decomposition}
		Let $v \coloneqq (v^{(1)}, \ldots, v^{(m)})$ be a unimodal $m$-decomposition of a given $m$-modal sequence. It holds that 
		\begin{equation}
			\begin{dcases}
				S^+(v^{(k)})= 2 & 1 < k < m\\
				1 \le S^+(v^{(k)}) \le 2 & \text{ otherwise}.
			\end{dcases}
		\end{equation}
		Furthermore, for each $\lambda_k \in \mr$ such that $S(v^{(k)}-\lambda_k)=2$, where $k \in [m]$, the sequence $v^{(k)}-\lambda_k$ displays the sign pattern $-+-$.
		\begin{proof}
			Suppose $u \coloneqq (u_1, u_2,\ldots)$ is an $m$-modal sequence with unimodal $m$-decomposition  $v\coloneqq (v^{(1)}, v^{(2)},\ldots)$. By Definition  \ref{def.multimodel}, $u$ has $2m$ monotonicity intervals $(0, a^-_1], [a^+_1, b^-_1], \ldots, [b^+_{m-1}, a^-_m]$ and $ [a^+_m, \infty)$. These are also monotonicity intervals of $v$. Due to the $m$-modality of $v$, $v^{(1)}$ is either monotone  on $(0, b^-_1)$ or changes direction at  $a^+_1$. Thus, for any $\lambda_1 \in (u_*,~u_{a^-_1})$, where $u_{a^-_1} > u_* > \min(u_1, u_{b^-_1})$,  we have $1 \le S(v^{(1)}-\lambda_1) \le 2$. Specifically, if $a^-_1=1$, the lower bound is attained; otherwise, the upper bound is attained. In this case, the sign pattern of the sequence $v^{(1)}-\lambda_1$ must be $-+-$.  Clearly, any other choice of $\lambda_1 \notin (u_*,~u_{a^-_1})$ cannot achieve the upper bound. The same argument applied to $v^{(m)}$ by symmetry.Here, $u_\infty$ plays the role of $u_1$ such that $u_\infty = u_{a^+_{m}+1}$ if  index $a^+_{m}+1$ exists and  $u_\infty = u_{a^+_{m}}$ otherwise.
			By (\ref{def.condition.multimodel.sequence1})-(\ref{def.condition.multimodel.sequence2}), each inner sequence $v^{(k)}$, where $k \in \{2, \ldots, m-1\}$ satisfies the inequality
			\begin{equation*}
				u_{b^+_{k-1}} < u_{a^-_{k}} > u_{b^-_{k}}.
			\end{equation*}
			As before, for every $\lambda_k \in (u_*, u_{a^-_{k}})$, where 
			$u_{a^-_{k}} > u_* > \max(u_{b^+_{k-1}}, u_{b^-_{k}})$, we have $S(v^{(k)}-\lambda_k) =2$. Moreover, the sign pattern displayed by $v^{(k)}-\lambda_k$ is $-+-$. This completes the proof.
		\end{proof}
	\end{assertion}
	Using Definition \ref{def.patition}, it can be shown that the number of sign changes in a given sequence is equal to the sum of the number of sign changes in any of its $m$-decomposition. This , however, does not hold for an $m$-partition of a sequence, as  the next lemma demonstrates.
	\begin{lemma}\label{lem.sign.chained.sequence}
		Fix $n_1, \ldots, n_m \in \mn$ and consider an $m$-partitioned sequence $u \coloneqq(u^{(1)}, u^{(2)}, \ldots, u^{(m)})$, where $u^{(k)}$ has $n_k$ elements with $n_k \ge 2$ terms when $k \in \{2,\ldots, m-1\}$ and $n_k \ge 1$ terms when $k \in \{1,m\}$.  Given  an $m$-decomposition of $u$ denoted $v \coloneqq (v^{(1)}, v^{(2)}, \ldots, v^{(m)})$ where
		\begin{alignat}{2}\label{def.partition}
			v^{(k)} \coloneqq \begin{dcases}
				(u^{(1)}_1, \ldots, u^{(1)}_{n_1},u^{(2)}_{1})& k =1,\\
				(u^{(m-1)}_{n_{m-1}},u^{(m)}_1, \ldots, u^{(m)}_{n_m})& k =m,\\
				(u^{(k-1)}_{n_{k-1}},u^{(k)}_1, \ldots, u^{(k)}_{n_k},u^{(k+1)}_{1})& 1 < k< m,			
			\end{dcases}
		\end{alignat}
		it holds that:
		\begin{enumerate}[label=\rm(\roman*)]
			\item $S(u)  = S(v) =  \sum_{i =1}^m S(v^{(k)})$.
			\item $ 0 \le S(u) -  \sum_{i =1}^m  S(u^{(k)}) \le  2m-2$.
		\end{enumerate}
	\end{lemma}
	\begin{proof}
		When $m=1$, there is nothing to prove. Now assume $m > 1$ and consider the  sequence $v^\prime \coloneqq (v^{(1)}, v^{(2)})$, formed by concatenating $v^{(1)}$  and  $v^{(2)}$. Since no sign change occurs between the last term of $v^{(1)}$ and the first term in $v^{(2)}$ (as these terms are identical), as the two terms in question are the same, concatenation does not alter the total number of sign changes.  Furthermore,  the duplicated term in $v^\prime$ can be discarded without affecting the total number of sign changes. If $m=2$, then the proof is complete. If $m > 2$,  proceed inductively,  appending each of $v^{(3)},  \ldots, v^{(m)}$, one at a time to $v^\prime$ until  $v$ is reconstructed. This procedure shows that $S(u) = S(v) =  \sum_{i=1}^m S(u^{(i)}).$

		For part \rm(ii), that transitioning from the last term of $u^{(1)}$ to  the first term in $u^{(2)}$ introduces 0 or 2 sign changes by Definition \ref{def.sign.chg}. With $m-1$ spaces between the components of $u$,  the maximum number of sign changes must be  $2(m-1)$, precisely as stated in \rm(ii).
	\end{proof}
	
	\begin{lemma}[Decomposition lemma]\label{lem.chained}
		Every $m$-modal sequence admits a unimodal $m$-partition (resp. $m$-decomposition). 
	\end{lemma}
	\begin{proof}
		Suppose $u \coloneqq (u_1, u_2, \ldots)$ is $m$-modal. By Definition \ref{def.multimodel}, $u$ has  exactly $2m-1$ distinct intervals   $[a^-_{1}, a^+_{1}] \prec \cdots \prec  [a^-_{m}, a^+_{m}]$ and $[b^-_{1}, b^+_{1}] \prec \cdots \prec  [b^-_{m-1}, b^+_{m-1}]$ such that the terms of $u$ satisfy relations (\ref{def.condition.multimodel.sequence1})-(\ref{def.condition.multimodel.sequence2}). Form $m$ contiguous subsequences of $u$, denoted $u^{(1)}, u^{(2)}, \ldots, u^{(m)}$, where $u^{(1)} \coloneqq (u_k)$ and $0 < k  \le b^-_1$, $u^{(m)} \coloneqq (u_k)$ and $ k > b^-_{m-1}$ ( $k  \ge  b^-_{m-1}$ for the $m$-decomposition) and finally for  $j \in (1, m)$, $u^{(j)} \coloneqq (u_k)$ for $j \in (b^-_{j-1}, b^-_{j}]$ ($j \in [b^-_{j-1}, b^-_{j}]$ for the $m$-decomposition). The proof is complete, as  each $u^{(k)}$ is unimodal.
	\end{proof}
	\begin{example}
		The  bimodal sequence $ u=(1, 5, 3, 4,2)$ can be partitioned into two  unimodal sequences $\{ (1,5,3), (4,2)\}$ or $ \{(1,5), (3, 4,2)\}$. The former corresponds to our algorithm in the decomposition lemma. The $m$-decomposition of $u$ yields $ \{(1,5,3), (3, 4,2)\}$, where the term 3 is shared between the two unimodal sequences. 
	\end{example}
	\begin{rem}
		$ $
		\begin{itemize}
			\item The converse statement of  Lemma \ref {lem.chained} is not true. Consider the monotone sequence  $(3, 2, 1,  0, -1)$ decomposed into two unimodal sequences $(3, 2, 1)$ and $(1, 0, -1)$.
			\item The construction of a unimodal $m$-partition  is not unique. 		However, a unimodal $m$-decomposition is  unique if,  for $k \in [m-1]$, the intervals $[b^-_{k}, b^+_{k}]$ are singletons.
		\end{itemize}
	\end{rem}
	\begin{lemma}[Concatenation lemma]\label{lem.chained2}
		Suppose that $u$ and $v$ are two nonconstant unimodal real sequences.  If the last term of $u$ equals the first term of $v$, then the concatenation of  $u$ and $v$ yields:			 
		\begin{itemize}
			\item A unimodal sequence if $u$ is  nonincreasing, or the second monotonicity interval of $u$ and the first monotonicity interval of $v$ are of the same type.
			\item A bimodal sequence if $u$ is not nondecreasing and the second monotonicity interval of $u$ and the first monotonicity interval of $v$ are of opposite types.
		\end{itemize}
	\end{lemma}
	\begin{proof}
		Let $u$ and $v$ be two nonconstant unimodal sequence with mode intervals  $[a_-,a_+]$ and  $[b_-, b_+]$ respectively. Suppose $u$ has $n$ terms.  After concatenation, the  combined sequence $w$, consisting of $u$ followed by $v$, has monotonicity intervals $[1, a_+ ]$,  $( a_+ , n]$,  $( n+1 , n+b_+]$ and $[n+b_+, \infty)$. By the assumption of part \rm(i), $w$ is nonincreasing  on $[1 , n]$ or  its monotonicity on $(a_+,n]$ and $[n+1 , n+b_+]$ is of the same kind. Denoting the monotonicity patterns by the symbols $ \nearrow, ~\searrow$ for being nonincreasing (or decreasing) and nondecreasing (or increasing), respectively,  the monotonicity pattern of $w$ is one of 
		$$\searrow\searrow\searrow\searrow, ~\nearrow\nearrow\nearrow\nearrow,~ \nearrow\nearrow\searrow\searrow,~ \nearrow\searrow\searrow\searrow,~ \nearrow\nearrow\nearrow\searrow.$$
		Since the last term of $u$ and the first in $v$ are identical,   no jump (or additional mode) is introduced when $w$ transitions from $u$ to $v$. Therefore, $w$ is either monotone  or is nondecreasing then nonincreasing, hence  unimodal.
		
		For part \rm(ii), the assumptions imply that $w$ has one of the following patterns
		$$\searrow\searrow\nearrow\nearrow,~ \nearrow\searrow\nearrow\searrow,~\nearrow\searrow\nearrow\nearrow.$$	
		Since $w$ has two distinct modes in the intervals $[a_-, a_+]$ and   $[n+b_-, n+b_+]$ and $ w_{a_+} > w_{n+1} < w_{n+b_n}$, it follows from (\ref{def.condition.multimodel.sequence1})-(\ref{def.condition.multimodel.sequence1}) that $w$ is bimodal.
	\end{proof}
	The next proposition establishes an upper bound on the maximum number of sign changes in a sequence, based on its mode-aligned transformation.
	\begin{assertion}\label{prop.sufficient.m-modal}
		Let $u$ be a nonconstant $m$-modal sequence and let $v \coloneqq (v^{(1)}, v^{(2)}, \ldots, v^{(m)})$ be  a unimodal $m$-decomposition of $u$. Denote by $v^\circ$  the mode-aligned transformation of $v$.	 The following hold: 
		\begin{enumerate}[label=\rm(\roman*)]
			\item $2m-2 \le S^+(v^\circ) = \sum_{k =1 }^m S^+(v^{(k)}) \le 2m$.		
			\item  $S^+(u) \le S^+(v^\circ)$.
		\end{enumerate}
	\end{assertion}
	\begin{proof}
		Let $v^\circ \coloneqq (v^{(1)}-\lambda_1, \ldots, v^{(m)}-\lambda_m)$ be the mode-aligned version of $v$, where  $\lambda_1, \ldots, \lambda_m$ are fixed numbers chosen to ensure that each component of $v^\circ$ has the same mode, as characterised in Definition \ref{def.mode-alignment}. Observe that $S(v^\circ) = \sum_{k = 1}^m S(v^{(k)}-\lambda_k)$, since transitioning from the last term of $v^{(k)}$ to the first  term of $v^{(k+1)}$ (for $k =1,2, \ldots,  m-1$) introduces no sign change. This equality extends to the maximum number of sign changes. Let $d$ the common mode of each component of $v^\circ$ and set $\alpha =  (d  + v_{(2)})/2$, where $v_{(2)}$ is the second largest value in $v^\circ$. This value  exists because $u$ is not a constant sequence. Therefore, $S(v^\circ-\alpha) = S^+(v^\circ)$ and  $S(v^{(k)}-\alpha) = S^+(v^{(k)})$ for all $k \in [m]$. Consequently, the bounds $2m$ and $2m-2$ are clear, as all the inner sequences of $v^\circ-\alpha$ must have exactly two sign changes each, while the outer sequences $v^{(1)}$  and $v^{(m)}$ may have one or two sign changes. This completes the proof of the first assertion. 
		
		Part \rm(ii) follows from$S(v-\lambda) = S(u-\lambda)$ for all $\lambda \in \mr$, since no sign change occurs between adjacent subsequences in an  $m$-decomposition. Because the modes of each $v^{(k)}$ may  differ,  $S^+(v) \le  \sum_{k=1}^m S^+(v^{(k)})$. Therefore, $S^+(u) = S^+(v) \le  \sum_{k=1}^m S^+(v^{(k)})$, as required.
	\end{proof}
	The intuition behind the previous proposition is that the number of intersection points created by a line crossing the graph of a multimodal function is maximised when the line lies below every mode but above all other values. This configuration is only guaranteed for an arbitrary multimodal sequence if all mode values coincide.
	\begin{rem}
		A more intuitive proof is as follows: Imagine representing the $m$ modes  as a string of $m$ plus sign. Notice that there are $m-1$ spaces  separating these signs and a total of $m+1$ potential locations, including the ends, where  minus signs could be placed.  An $m$-modal sequence's pattern can thus contain anywhere from $m-1$ to $m+1$ minus signs.  Crucially, at least $m-1$ minus signs are required to ensure every pair of plus signs is separated by at least one minus sign. This guarantees at least $2m-2$ sign changes and at most $2m$ sign changes.
	\end{rem}
	Now, we are in a position to state the main results of this section.
	\begin{thm}\label{thm.multimodal}
		Fix $m \ge 1$. A  regular sequence $u$ is $m$-modal iff it has a max-adjusted $m$-decomposition $u^\circ$ such that the set
		$A_l = \{\alpha \in \mr: S(v^\circ-\alpha) =2m - l\}$ is not empty for some 
		$l \in  \{0, 1, 2\}$. When $A_l$ is not empty (for $l \in  \{0, 1, 2\}$),  the sign arrangement of the sequence $(v^\circ-\lambda_l)$, where $\lambda_l \in A_l$, occurs in the same order. Specifically, the pattern consists of $m$ plus signs and $m -l$ minus signs arranged in an alternating fashion.
	\end{thm}
	To prove this theorem, we require the following auxiliary result.
	\begin{thm}\label{thm.sign}
		Consider a string of plus and minus signs distributed across $m \ge 1$ cells such that the signs at the boundaries of adjacent cells match. The total number of sign changes is:
		\begin{enumerate}[label=\rm(\roman*)]
			\item $2m$  iff each  cell has 2 sign changes.
			\item $2m-1$  iff  every cell except one of the outer cells has two sign changes.
			\item $2m-2$  iff each inner cell has two sign changes, and both outer cells have one sign change each.
		\end{enumerate}
		Furthermore, there are $m$ plus signs and the remaining signs  are minus signs, or $m$ minus signs and the remaining signs are plus signs.
	\end{thm}
	\begin{proof}
		Suppose the observed sequence has $2m$ sign changes. By hypothesis, each cell contains at least one  plus and one minus sign. Repeat the pattern $+-$ or $-+$, $m$ times, placing a vertical bar $|$ after each occurrence to form $m$ cells: $+-|+-| \cdots|+-|+-$ or $-+|-|-+| \cdots|-+|-+$. Since adjacent cells share a sign in common, scan the string of signs from  left to right and insert the last sign of each cell at the beginning of the next. This gives the  arrangement
		\begin{equation*}
			-+-|-+-|-+-| \cdots|-+-|-+-
		\end{equation*}
		or
		\begin{equation*}
			+-+|+-+|+-+| \cdots|+-+|+-+
		\end{equation*}
		Each inner cell has exactly two sign changes, totalling 2m sign changes. Now, we show that no other configuration of signs satisfies the theorem's hypothesis. 
		
		Modifying this arrangement involves rearranging the vertical bars. Each separator can only move left or right by one or two positions. Moving a separator by one position is impossible, as it creates a sign inconsistency on either side of the separator. For example  $|-+-|-+-|$ to  $|-+|--+-|$. Moving a separator by two positions is also impossible, as it creates a singleton cell with only one type of sign. Therefore, no other configuration is possible, and each cell must have two sign changes. Assertions \rm(ii) and \rm(iii) are proved similarly, considering combinations where the first sign in the first cell  and the last sign in the last cell are removed.

		For sufficiency, suppose  each inner cell has $2m-l$ sign changes. We may arrange these as
		\begin{equation}\label{def.sign.bimodal}
			\bullet-+|+-+|+-+| \cdots|+-+|+-\,\circ,
		\end{equation}
		or
		\begin{equation}\label{def.sign.unimodal}
			\blacklozenge +-|-+-|-+-| \cdots|-+-|-+\,\lozenge,
		\end{equation}
		where $\circ, \, \bullet$ are each either a plus sign or a blank space  and $\lozenge, \, \blacklozenge$  are each either a minus sign or a blank space. Both configurations (\ref{def.sign.bimodal}) and (\ref{def.sign.unimodal}) are $2m-2$ excluding the end symbols. Varying these symbols between blank and sign states yields $2m-1$ or $2m$ total sign changes.  For the last assertion, it is clear that (\ref{def.sign.bimodal}) has $m$ minus signs and (\ref{def.sign.unimodal}) has $m$ plus signs.
	\end{proof}
	\begin{proof}[Proof of Theorem \ref{thm.multimodal}]
		The requirement that each subsequence $(u_k)_{k \ge }$ has a finite maximum for every $n \in I$, stems from the existence condition imposed by Theorem \ref{thm.multimodal.maximum2}. If $u$ is constant,  it is  trivially unimodal. Now assume $u$ is nonconstant. For the necessary part, suppose $u$ is $m$-modal. By Lemma \ref{lem.chained}, $u$ has a unimodal $m$-decomposition $(v^{(1)}, \ldots, v^{(m)})$, which satisfies  by Lemma \ref{lem.numsign.decomposition}, the property $S^+(v^{(k)}) =2$ when $k=2,\ldots, m-1$ and  $1 \le S^+(v^{(k)}) \le 2$ when $k=1$ or $m$. Let $v^\circ \coloneqq (v^{(1)}-\lambda_1, \ldots, v^{(m)}-\lambda_1)$ be the mode-adjusted  transformation of $v$, where $\lambda_1, \ldots, \lambda_{m}$ are selected as prescribed in  Definition \ref{def.mode-alignment}. By the same lemma, each inner sequence $v^{(k)}-\lambda_k$, where $k=2, \ldots, m-1$,  have  the sign pattern  $-+-$. Consequently, the sign pattern of $v^{(1)}-\lambda_1$ must end with a minus sign,  and  $v^{(m)}-\lambda_m$ must begin with one to match  the signs of their adjacent sequences $v^{(2)}$ and $v^{(m-1)}$ respectively. Because $v$ is not constant,  $v^{(1)}-\lambda_1$ and $v^{(m)}-\lambda_m$ must change sign at least once. Therefore, they each have a plus sign and possibly another minus sign, as $2m-2 \le S^+(v^\circ) \le 2m$ in view of Proposition \ref{prop.sufficient.m-modal}. Consequently, for $l \in \{ 0, 1, 2\}$  and  $\lambda \in \{\alpha \in \mr: S(v^\circ-\alpha) =2m - l\}$,  $v^\circ-\lambda$ has $m$ plus signs interlaced with  $m-l$  minus signs,  counting duplicated signs only once.
		
		For sufficiency, suppose that for a given $m$, there exists an  $l \in  \{0, 1, 2\}$ such that $A_l = \{\alpha \in \mr: S(v^\circ-\alpha) =2m - l\}$, and  the sign pattern of $(v^\circ-\lambda_l)$, where $\lambda_l \in A_l$, remains the same, consisting of  exactly $m$ plus signs and $m-l$ minus signs. Applying Theorem \ref{thm.sign}, we deduce that each inner sequence has precisely two sign changes, with the remaining sign changes distributed between the outer sequence $v^{(1)}$ and $v^{(m)}$. Only the pattern (\ref{def.sign.unimodal}) satisfies the condition of $m$ plus signs interlaced with $m-l$ minus signs (excluding duplicated adjacent signs). Therefore, each $v^{(k)}$ is unimodal and by application of Lemma \ref{lem.chained2}, $v$ and thus $u$ is unimodal.
	\end{proof}
	
	A special version of this theorem appeared  recently in \cite{Karp}  (see Lemma 2.4), where a  necessary and sufficient  condition for unimodality was proposed. However, the  condition suggested is insufficient, as it overlooked two scenarios: sequences lacking a maximum and bimodal sequences with the sign pattern $+-+$. The following corollary, valid for sequence, provides a correction. A similar result for functions can be derived.
	\begin{cor}\label{lem.sign.chg}
		A regular sequence $u$  is
		unimodal iff $S^+(u) \le 2$, and for all $\lambda \in \mr$ such that $S(u-\lambda)=2$, the sign pattern of $(u-\lambda)$ is necessarily $-, +, -$.
	\end{cor}
	\begin{proof}
		This assertion follows directly from Theorem \ref{thm.multimodal} when $m=1$. 
	\end{proof}

	The definition of an $m$-modal sequence includes $m-1$ intervals of local minima. Considering the additive inverse of an $m$-modal sequence, then by symmetry, these intervals of local minima, $[b^-_1, b^+_1], \ldots, [b^-_{m-1}, b^+_{m-1}]$, turn into  intervals of modes. The following lemma illustrates this symmetrical relationship.
	\begin{lemma}\label{lem.bimodal}
		A real sequence $u$ with exactly two sign changes is either nonmodal, unimodal or bimodal. 
		Specifically, $u$ is unimodal or bimodal if  $u$ is regular. If, in addition, $-u$ is regular, then the unimodality of $u$ is equivalent to the bimodality of  $-u$.
	\end{lemma}
	\begin{proof} 
		Assume  $u$ and $-u$ satisfy the existence condition for $m$-modaltiy as described in the lemma's hypothesis. If the two tail conditions are not satisfied, then, by Theorem \ref{thm.multimodal.maximum2},  $u$ and $-u$ are nonmodal.  By hypothesis, there exists a $\lambda \in \mr$ such that $S(u-\lambda)=2$. Therefore,  the sign pattern of  $u-\lambda$ is either $+-+$ or $-+-$.  Since each $+$ represents a distinct mode, $u$ is unimodal in the first case and bimodal in the second. If $u$ is unimodal, then  $-u$ is bimodal, as its sign pattern must be $+-+$. Conversely, if $u$ is bimodal with sign pattern $+-+$, then $-u$ can only be  unimodal, which is  the desired result.
	\end{proof}

	\begin{figure}[h]
		\centering
		\begin{minipage}{0.45\textwidth}
			\begin{tikzpicture}
				\begin{axis}[
					width=8cm,
					height=5.5cm,
					xlabel={A bimodal sequence $u$ with  two sign changes},
					ylabel={},
					title={},
					legend pos=south west,
					legend style={font=\small},
					tick label style={font=\small},
					label style={font=\small},
					title style={font=\small},
					grid=none
					]
					\addplot table [x=index, y=value, col sep=comma] {bimodal.csv};
					\addplot table [x=index, y=lambda, col sep=comma, mark={}] {bimodal.csv};				
					
				\end{axis}
			\end{tikzpicture}
			
		\end{minipage}
		\hfill
		\begin{minipage}{0.45\textwidth}
			\begin{tikzpicture}
				\begin{axis}[
					width=8cm,
					height=5.5cm,
					xlabel={The additive inverse of $u$ is a unimodal sequence},
					ylabel={},
					title={},
					legend pos=south west,
					legend style={font=\small},
					tick label style={font=\small},
					label style={font=\small},
					title style={font=\small},
					grid=none
					]
					\addplot table [x=index, y=value2, col sep=comma] {bimodal.csv};
					\addplot table [x=index, y=lambda2, col sep=comma, mark={}] {bimodal.csv};				
				\end{axis}
			\end{tikzpicture}
		\end{minipage}			
		\vskip\baselineskip 
		
		\begin{minipage}{0.45\textwidth}
			\begin{tikzpicture}
				\begin{axis}[
					width=8cm,
					height=5.5cm,
					xlabel={A trimodal sequence $v$ with  four sign changes},
					ylabel={},
					title={},
					legend pos=south west,
					legend style={font=\small},
					tick label style={font=\small},
					label style={font=\small},
					title style={font=\small},
					grid=none
					]
					\addplot table [x=index, y=value, col sep=comma] {trimodal.csv};
					\addplot table [x=index, y=lambda, col sep=comma, mark={}] {trimodal.csv};				
					
				\end{axis}
			\end{tikzpicture}
		\end{minipage}
		\hfill
		\begin{minipage}{0.45\textwidth}
			\begin{tikzpicture}
				\begin{axis}[
					width=8cm,
					height=5.5cm,
					xlabel={The additive inverse of $v$ is a bimodal sequence},
					ylabel={},
					title={},
					legend pos=south west,
					legend style={font=\small},
					tick label style={font=\small},
					label style={font=\small},
					title style={font=\small},
					grid=none
					]
					\addplot table [x=index, y=value2, col sep=comma] {trimodal.csv};
					\addplot table [x=index, y=lambda2, col sep=comma, mark={}] {trimodal.csv};				
					
				\end{axis}
			\end{tikzpicture}
		\end{minipage}			
		\caption{$m$-modal sequences and their additive inverses.}		
		\label{fig.multimodal}
	\end{figure}
	A generalisation of this lemma follows
	\begin{thm}\label{prop.reversal.multimodal}
		Let  $u$ be a nonconstant $m$-modal sequence, and  $v \coloneqq (v^{(1)}, v^{(2)}, \ldots, v^{(m)})$ its  unimodal $m$-decomposition. Denote by $v^\circ$  the mode-aligned transformation of $v$. If $-u$ is regular, then $-u$ is $m$-modal, $(m+1)$-modal or $(m-1)$-modal. Specifically,
		\begin{enumerate}[label=\rm(\roman*)]
			\item If $S^+(v^\circ)=2m$, then $-u$ is $(m+1)$-modal.
			\item If $S^+(v^\circ) = 2m-2$, then $-u$ is $(m-1)$-modal.
			\item If $S^+(v^\circ) = 2m-1$, then $-u$ is $m$-modal.
		\end{enumerate}
		If $-u$ is nonregular, then $-u$ is nonmodal.
	\end{thm}
	\begin{proof}
		By Theorem \ref{thm.multimodal}, $u$ possesses a mode-aligned $m$-decomposition $v^\circ \coloneqq (v^{(1)}-\lambda_1, \ldots, v^{(m)}-\lambda_m)$, where $(\lambda_1, \ldots, \lambda_m) \in \mr^m$ are derived in Definition \ref{def.mode-alignment}. In view of Proposition \ref{prop.sufficient.m-modal}, there exists a $\lambda \in \mr$, such that $2m-2 \le S(v^\circ-\lambda) \le 2m$. The sign pattern of   $v^\circ-\lambda$ is
		\begin{equation}\label{def.sign.pattern.rev}
			\circ\circ\circ|-+-|-+-| \cdots|-+-|\bullet\bullet\,\bullet,
		\end{equation}
		where $\circ\circ\circ$ is either $+-$ or $-+-$ and $\bullet\bullet\,\bullet$ is one of   $-+$ or $-+-$. Inverting the signs yields the sign pattern of $\lambda-v^\circ$,  given by
		\begin{equation}\label{def.sign.pattern}
			\lozenge\lozenge\,\lozenge|+-+|+-+| \cdots|+-+|\blacklozenge\blacklozenge\,\blacklozenge,
		\end{equation}
		where  $\lozenge\lozenge\lozenge$ is either $-+$ or $+-+$ and $\blacklozenge\blacklozenge\,\blacklozenge$ is one of   $+-$ or $+-+$.	By symmetry, $-v^\circ$ is a mode-aligned $m$-decomposition of $-u$ and $-\lambda$
		maximises the number of sign changes in $\lambda-v^\circ$. 	 The inner sequences in (\ref{def.sign.pattern.rev}) have $m-2$ plus signs with $2m-4$ sign changes. In contrast, the inner sequences in (\ref{def.sign.pattern}) have  $m-1$ plus signs and the same number of sign changes as (\ref{def.sign.pattern.rev}). Thus, three scenarios arise. First, if $S^+(v^\circ)=2m$, the arrangement of signs $\circ\circ\circ$ and $\bullet\bullet\bullet$ must be $-+-$ in order to add four sign changes to the total. Second, if $S^+(v^\circ)=2m-2$, $\circ\circ\circ$ and $\bullet\bullet\bullet$, are respectively $+-$ and $-+$ as we require two sign changes. Third, if $S^+(v^\circ)=2m-1$, $\circ\circ\circ$ and $\bullet\bullet\bullet$ correspond  respectively  to $+-$ and $-+-$ or to $-+-$ and $-+$ in order to add three sign changes to the total.
		
		In the first scenario, inverting $\circ\circ\circ$ and $\bullet\bullet\bullet$ add one plus sign each, giving  $m+1$ plus signs to  $\lambda-v^\circ$, so $-u$ is $(m+1)$-modal. By a similar argument, in the second scenario, $\lozenge\lozenge\,\lozenge$ and $\blacklozenge\blacklozenge\,\blacklozenge$  add no new plus signs. In this case, $\lambda-v^\circ$ has $m-1$ sign change, so $-u$ is $(m-1)$-modal. In the third scenario, $\lozenge\lozenge\,\lozenge$ and $\blacklozenge\blacklozenge\,\blacklozenge$  together add  one additional plus sign; hence, $-u$ is $m$-modal. Finally, if  $-u$ is nonregular, then $-u$ is nonmodal by application of Theorem \ref{thm.multimodal.maximum2}. 
	\end{proof}
	
	The intuition behind Proposition \ref{prop.reversal.multimodal} and its special case, Lemma \ref{lem.bimodal}, stems from the patterns $\wedge$ and $\vee$ in the unimodal/bimodal case and $M$/$W$ in the bimodal/trimodal setting. These configurations are depicted in Figure~\ref{fig.multimodal}, where the mode-adjusted $m$-decomposition has $2m-2$ sign changes (same as the original sequence), and the additive inverse corresponds to an $(m-1)$-modal sequence.
	We close this section with an easy lemma for determining the number of sign changes associated with a given $m$-modal function.

	In the following section, we turn our attention to the second topic of this study: transformations of $m$-modal sequences using sign-regular kernels.
	\section{Preservers and Reversers of Unimodality}\label{sec2}
	\subsection{The impact of sign-regularity on unimodal sequences}\label{sec2.1}
	Let $I, J$ be convex subsets of $\mr$,   which can be intervals of the real line or countable set of discrete values along the real line. Let $K(x,y)$ be a kernel, which is a bounded function, defined on $I \times J$ such that $\int_J K(x,y)\, \dx[y] < \infty$. Define the  transformation $\eL$ based on kernel $K$, where for a bounded function  $f:J\to \mr$, as follows:
	\begin{equation*}
		g(x) \coloneqq (\eL f)(x) = \int_{J} K(x,y)f(y)\,\dx[y],~~ x \in I.
	\end{equation*}
	As our focus is on  sequence-to-sequence transformations, we will assume that $I\subseteq \mn$ and $J \subseteq \mn$ and the integration is substituted by summation. We shall use the symbols $u, v$ for sequences instead of $f, g$. This yields the transformation 
	\begin{alignat}{2}
		v_n &\coloneqq  \sum_{j \in J} K(n,j)u_j,\qquad &n \in I \label{def.kernel.sequence-sequence}.
	\end{alignat}
	Unless otherwise stated, we  shall assume that $K(\cdot,j)$ attains all of its values for each $j \in J$ and that the sum (\ref{def.kernel.sequence-sequence}) is finite and exists ($v$ is assumed to converge uniformly on $J$). 
	
	We are interested in determining conditions on  a transformation $\eL$ that maps an $m$-modal sequence to either an $m$-modal sequence or to a sequence whose additive inverse is $m$-modal. We shall rely on the theory  of total positivity to provide simple, sufficient conditions ensuring that $\eL$ preserves or reverses $m$-modality.   We begin by recalling a  fundamental theorem of Karlin (see p.233 in \cite{KarlinTPBook}), which states that under certain conditions, transformations based on a $\SR{r}$ kernel preserve or reverse the sign pattern of functions with at most $r-1$ sign changes.
	\begin{thm}[Theorem 3.1 \& Proposition 3.1,  Chapter 3 in \cite{KarlinTPBook}]\label{def.vd}
		Fix $r \ge 1$ and let $I, J$  be two intervals in $\mn$ with at least $r$ points. Suppose  $\eL$  is a transformation  defined in (\ref{def.kernel.sequence-sequence}), using a kernel $K: I\times J \to \mr$. If $K$ is $\SR{r}$, then for a  real-valued sequence $u$ defined on $I$ and its transformation 	$v\coloneqq\eL{u}$,  the following hold:
		\begin{enumerate}[label=\rm(\roman*)]
			\item  $\eL$ is VD if $S({u}) \le r-1$.  Specifically, if $r=2$ and $u$ is monotone, then $v$ is monotone.
			\item If $S({v})=S({u}) = k \le r-1$, then ${u}$ and ${v}$ exhibit the same monotonicity pattern if  $K$ is $\SR{r}$, 
			and the signs of its minors of order $k$ and $k+1$ satisfy $\varepsilon_{k}\varepsilon_{k+1}=1$. However, if  these two signs oppose each other, then ${v}$ exhibits the reverse sign pattern to  $u$.  					
		\end{enumerate}
	\end{thm}

	This theorem implies that any transformation using a $\TP{r}$ kernel preserves the sign changes of a sequence $u$, provided the transformed sequence has the same number of sign changes, say $k \le r-1$. For $\SR{r}$ kernels, the transformation may preserve or reverse these sign changes, depending on whether the signs of the minors of order $k$ and$ k+1$ ( denoted  $\varepsilon_{k}$ and $\varepsilon_{k+1}$ respectively)  are the same or different.
	
	A key application of the preceding result involves the convexity and concavity of functions. Specifically, Karlin (Proposition 3.2, p. 23 in \cite{KarlinTPBook}) proved that if a kernel  is  $\TP{3}$, then it preserves the convexity and concavity of functions. However, this preservation property does not automatically extend to  $\SR{3}$ kernels, as Karlin also noted in p. 23 in  the same reference. 
	
	Karlin showed that a $\TP{3}$ kernel preserves convexity and concavity (Proposition 3.2, p. 23 in \cite{KarlinTPBook}). However, this doesn't automatically extend to $\SR{3}$ kernels.  He previously claimed that a strictly $\SR{2}$ (no determinant of order 1 or 2 vanishes)  kernel unconditionally preserves the monotonicity direction of a function (see p.343 in \cite{KarlinConvexity}). This claim is incorrect.  While his monograph mentions monotonicity preservation (as in part (i) of the theorem), it omits the earlier result in \cite{KarlinConvexity}. A strictly $\SR{2}$ kernel is still $\SR{2}$ , and by part (ii) of the theorem, the sign patterns of u and v must be opposite. Thus, v would be concave if $u$ is convex, for example.

	The following discussion of higher-order convexity will be essential for understanding the subsequent results on $m$-modality preservation.

	Let $\Delta$ be the difference operator defined on a sequence $u\coloneqq (u_1, u_2, \ldots)$, where $\Delta u \coloneqq (u_2-u_1, u_3-u_2, \ldots)$ and $\Delta^m u = \Delta^{m-1} \Delta u$ for $m > 1$.
	\begin{define}[$m$-convex/concave sequence]\label{def.convex}
		A sequence $u$ is $m$-convex if $\Delta^m u \ge0$ (see p.23 in \cite{KarlinTPBook}). 		More generally, a sequence $u$ is said to be $m$-convex or convex of order $m$ if for an arbitrary polynomial of degree $m - 1$, $p(x) = a_0 x^{m-1} + a_1 x^{m-2} + \cdots + a_{m-1}$, such that $v \coloneqq (v_k)$, where $v_k \coloneqq u_k - p(k)$,  has at most $m$ changes of sign in $I$, and if $m$  sign changes actually occur, they follow the pattern $+-+\cdots$. A sequence $u$ is $m$-concave if $-u$ is $m$-convex. 
	\end{define}
	Based on this definition, a nondecreasing sequence  is 1-convex or convex of order 1. It is convex (2-convex or convex of order 2) when $\Delta^2 u \ge 0$.

	We introduce a second definition of an $m$-modal sequence, inspired by the $m$-convexity definition and by Proposition \ref{prop.sufficient.m-modal}.
	\begin{define}[$m$-modal sequence]\label{def.mmodal}
		A regular sequence $u$ is $m$-modal if,  for arbitrary polynomials $p_1, p_2, p_3$ of degree $2m - 1$, $2m-2$ and $2m-3$, respectively, $u-p_1$  has at most $2m$ sign change; if no such $p_1$ exists, then $u-p_2$ has at most $2m-1$ sign changes; and  If neither $p_1$ nor $p_2$ exist, then $u-p_3$ has at most $2m-2$ sign changes. In each case, the resulting sign pattern must contain $m$ plus signs.
	\end{define}

	\begin{define}[$m$-convexity preserving/reversing kernels]\label{def.mpp}
		Fix $m \ge 1$. Let $\eL:I \times J \to \mr$  be a kernel, where $I, J$ are two ordered sets in $\mn$ with at least $r$ points. $\eL$ is an $m$-\textit{convexity-preserving} kernel, if, for any real-valued $m$-convex (or $m$-concave) sequence $u$ defined on $I$, $\eL u$ is $m$-convex (resp. $m$-concave).
	\end{define}
	The terms convexity-preserving and convexity-reversing are general and include  the property of preserving and reversing concavity, respectively. Reversing convexity is based on reversing the sign of the sequence after transformation. This reflects the fact that if $u$ is convex, then $-u$ is concave. With this sign-reversal concept in mind, we extend the notion of a sign-reversing kernel to $m$-modal sequence as follows.

	\begin{define}[$m$-modality preserving/reversing transformation]\label{def.um}
		$\eL$ is $m$-\textit{modality preserving}, if for all real  $m$-modal sequences $u$ for which $\eL u$ is defined, $\eL u$ is $m$-modal. Conversely, $\eL$  is  $m$-\textit{modality reversing}  if $-\eL u$ is $m$-\textit{modal}. In view of Theorem \ref{prop.reversal.multimodal}, this means that $\eL$ is  $m$-{modality reversing} if there exists an $m$-modal sequence  $\overline{u}$ such that $\eL \overline{u}$ is  nonmodal,  $(m-1)$-{modal}  or $(m+1)$-{modal}.
	\end{define}
	From these definitions, the following properties are clear:
	\begin{lemma}\label{lem.mpp.properties}
		Consider the linear transformation given by  (\ref{def.kernel.sequence-sequence}), the following assertions hold:
		\begin{enumerate}[label=\rm(\roman*)]
			\item If a transformation is $m$-modality preserving (resp. reversing), then its additive inverse  is $m$-modality  reversing (resp. preserving).
			\item The composition  of two  $m$-modality preserving or  two  $m$-modality reversing   transformations  is   $m$-modality preserving.	
			\item The composition of  two transformations, one $m$-modality preserving and the other  $m$-modality reversing  is $m$-modality reversing.
		\end{enumerate}
	\end{lemma}
	\begin{proof}
		Suppose that $\eL_1, \eL_2$ are $m$-modality preserving and let $u$ be any $m$-modal sequence for which these kernels are defined. Let $\negL{L}_1 \coloneqq -\eL_1$ and $\negL{L}_2 \coloneqq -\eL_2$. 	
		The proofs are straightforward. The first follows directly from the definition of $m$-modal kernels. The second demonstrates that the $m$-modality-preserving property is closed under composition. Given $v_1 = \eL_1 u,~ v_2 = \eL_2u$ are both $m$-modal, then  $\eL_1\eL_2(u) = \eL_1v_2$ and $\eL_2\eL_1(u) = \eL_2v_1$, are also $m$-modal. From this, we deduce that  both $\negL{L}_1\negL{L}_2$ and $\negL{L}_2\negL{L}_1$ are $m$-modality preserving,  as $\negL{L}_1\negL{L}_2(u) = \eL_1\eL_2(u)$ and  $\negL{L}_2\negL{L}_1(u) = \eL_2\eL_1(u)$. Finally, $\eL_1\negL{L}_2(u) = -\eL_1\eL_2(u)$, and  conversely $\eL_2\negL{L}_1(u) = -\eL_2\eL_1(u)$. Since $\eL_1\eL_2$ is $m$-modality preserving by property \rm(ii), its additive inverse, $-\eL_1\eL_2$,  is   $m$-modality reversing  by property \rm(i) and the desired result is obtained.
	\end{proof}
	We now focus on  the special case $m=1$ of $m$-modal sequences and establish a sufficient condition for a transformation to be unimodality-preserving/reversing. We first require the following auxiliary result before presenting the main theorem.
	\begin{lemma}\label{lem.bound}
		Suppose that $\eL$ is a transformation defined in (\ref{def.kernel.sequence-sequence}) based on a $\TP{1}$  kernel $K: I\times J \to \mr$, where $K(n,\cdot)$ attains all of its values in $J$. Let $v \coloneqq \eL u$, where $u$ is a real sequence.  If  $(u_k)_{k \ge l}$ attains its maximum (resp. minimum) for each $l \in I$, then $(v_k)_{k \ge n}$ attains its maximum (resp. minimum) for each $n \in J$. If, however, $(u_k)_{k \ge l}$ does not attain its maximum (resp. minimum) for some $l \in I$, then there exists a sign-regular kernel  $M$ for which  $(v_k)_{k \ge n}$ where $v_k \coloneqq \sum_{j \in J} M(k,j)u_j$,  does not attain its maximum (resp. minimum) on $J$. 
	\end{lemma}
	\begin{proof}	
		Since $K$  is $\TP{1}$, it is  nonnegative and $K(i,j)$  attains all of its values as $i$ and $j$ traverse $I$ and $J$ respectively. If $u$ has a maximum, then the sum of products $\sum_{j \in J} K(i,j)u_j$ also has a maximum. Conversely, if $u$ has a minimum, the sum of products attains its minimum. Thus, $v$ has a maximum in the first case and a minimum in the second. If there exists an $l \in I$ such that $(u_k)_{k \ge l}$ has no maximum  (resp. no minimum), consider the identity or permutation transformation, whose kernels are $\TP{}$ and $\SR{}$, respectively. Since the terms of $v$ and $u$ are identical in both cases,  the proof is complete.
	\end{proof}
	
	\begin{thm}\label{thm.unimodality.preserving} 
		Let $\eL$  be a transformation based on a kernel $K$  as defined in (\ref{def.kernel.sequence-sequence}). Suppose $K$ is $\SR{3}$. Then  $\eL$ is  unimodality-preserving iff 
		$\varepsilon_{1}=1$ and $\varepsilon_{2}=\varepsilon_{3}$. Conversely,   $\eL$ is  unimodality-reversing iff $\varepsilon_{1}=-1$ and $\varepsilon_{2}=-\varepsilon_{3}$.
	\end{thm}
	
	\begin{proof}
		Let $u_0, u_1, u_2$ be three unimodal sequences with 0, 1 and 2 sign changes, respectively, and let $v_i = \eL u_i$ for $i=0,1,2$. For the necessity part, suppose $\varepsilon_1=1$ and  $\varepsilon_{2}=\varepsilon_{3}$. If $u_0$ is constant, then $v_0$ is also constant by the VD property of $K$ and therefore unimodal. Similarly, $v_1$, which attains all of its values by Lemma \ref{lem.bound},  must be monotone or constant given the VD property of $K$ and is thus unimodal. Finally,  the VD property of $K$ implies that $v_2$ has at most two sign changes. If $S(v_2) \le 1$,  the proof is complete as $v_2$ is monotone and attains all of its values. If  $S(v_2)=2$, then  part \rm(ii) of Theorem \ref{def.vd}  establishes  that  $u_2$ and $v_2$ have the same sign pattern  $-+-$, making $v_2$ unimodal. Since $\eL$ maps arbitrary unimodal sequences to unimodal sequences, it is unimodality-preserving.

		To prove sufficiency, we proceed by contraposition. Suppose  $\varepsilon_{1}=-1$ or $\varepsilon_{2} \ne \varepsilon_{3}$. First, assume  the latter holds. In this scenario if $S(v_2) =S(u_2)=2$,  then $v_2$ is bimodal by Lemma \ref{lem.bimodal}, so, $\eL$ is not unimodality-preserving. Now, consider the case $\varepsilon_{1}=-1$. Suppose $S(v_1) =S(u_1)=1$ and $u_1$ does not attain its minimum. Then $-u_1$ does not attain its maximum. Since $-K$ is $\TP{1}$,  the sequence $(v^{(1)}) \coloneqq (v^{(1)}_k)$, where $ v^{(1)}_k =\sum_{j \in J}K(k,j) u^{(1)}_j = \sum_{j \in J}-K(k,j) (-u^{(1)}_j)$ does not attain its maximum for some choice of $-K$ by Lemma \ref{lem.bound}. Choose $K=-I$, where $I$ is the identity matrix, which satisfies the conditions $\varepsilon_1=-1$ and $\varepsilon_2=-\varepsilon_2$. Thus, $\eL$ is not unimodality-preserving, proving the contraposition. 
		
		For the second part of the theorem, recall that if $M$ is an $n\times n$ matrix whose determinant is $d$, then the determinant of  $-M$ is $(-1)^n d$. The unimodality-reversing result follows by applying  the first part to the transformation $-\eL$ associated with the kernel  $-K$. This completes our proof.
		
	\end{proof}
	
	From this point forward, we will use UP to abbreviate an $\SR{3}$ unimodality-preserving kernel or transformation and UR for an $\SR{3}$ unimodality-reversing kernel or transformation.
	
	\subsection{Preserving Unimodality for Differences, Ratios and Sums}
	One of the ramifications of Theorem \ref{thm.unimodality.preserving}, combined with the VD property,  is that if $v$ is a sequence resulting from  a UP transformation,  it cannot change direction before the input sequence. This is illustrated in the following lemma.

	\begin{lemma}\label{lem.modelocation}
		Let $a,b,c,d \in \mn$. Suppose that $u$ is an arbitrary sequence and let $v \coloneqq   \eL u$, where $\mathcal{L}$ is UP. If $u$ is unimodal with mode interval $[a,b]$, then $v$ is unimodal with a mode interval $[c,d]$	such that
		$$a \le c \le b \le d ~\text{ or }~ a \le b \le c \le d.$$
	\end{lemma}	
	\begin{proof}
		Since both $u$ and $K$ are nonconstant, $v$ is also nonconstant. Because  $\eL$ is a UP transformation, $v$ is unimodal. Let its mode interval be $[c, d]$. If $u$ is monotone, then the VD property implies  $v$ is also monotone, and both have at most one sign change. Since $S(u-u_{b})=0$, clearly $S(v-\eL u_{b})=0$. Thus, the first and last mode indices of $v$ must be on or after $[a, b]$. 
		
		Now, suppose  $u$ is not monotonic.  Let $\lambda=u_a / 2 + \max(u_{a-1}, u_{b+1}/2)$, with $u_{0}=u_1$ and $u_{N+1}=u_N$ if $u$ is finite with $N$ terms. It is easy to see that  $S(u-\lambda)=2$, and $S(u-u_{a})=S(u-u_{b})=0$. By the VD property, $v$ has zero sign change  up to index $b$ and  may change direction at $b+1$ as $S(v-\eL \lambda) \le 2$. Thus, the mode index of $v$ can not be less than $b$, and  the first mode index of $v$ must be on or after $[a, b]$. If the last mode index of $v$ satisfies $a < c  \le d < b$  then at index $d+1 \le b$, $v$ starts to decrease, so $v-\lambda$, where $\lambda \coloneqq \max(v_{d+1}, v_{c-1})/2 + v_{d}/2$; assuming $v_{0}=v_1$ and $v_{N+1}=v_N$ when $v$ is finite with $N$ terms, changes its sign from $+$ to $-$. But this is not possible by the VD property. This proves the  lemma.
	\end{proof}

	We now turn our attention to recurrence relations associated with unimodal sequences formed from ratios and differences of other sequences.
	\begin{assertion}\label{prop.recursion.up.diff.unimodal}
		Let $\eL$ be a linear transformation based on a matrix $K$ as defined in (\ref{def.kernel.sequence-sequence}).  For $n \in \mn$, define the sequence $v_{(n)} \coloneqq \eL^n u$,  
		where $u$ is an arbitrary nonconstant real sequence. Given that $u-v^{(1)}$ is unimodal, the following hold:
		\begin{enumerate}[label=\rm(\roman*)]
			\item If $\mathcal{L}$ is UP  then $v_{(n)} - v_{(n+1)}$ is unimodal.
			\item If $\mathcal{L}$ is UR then  $(-1)^n \left\{ v_{(n)} - v_{(n+1)}\right\}$ is unimodal.
		\end{enumerate}
		Given  that $K$  positive, $u$  nonnegative and $u/v_{(1)}$ unimodal, the following assertions hold:
		\begin{enumerate}[label=\rm(\roman*), start=3]
			\item If $\mathcal{L}$ is UP then $v_{(n)} / v_{(n+1)}$ is unimodal.
			\item If $\mathcal{L}$ is UR  then $(-1)^n v_{(n)} / v_{(n+1)}$ is unimodal.
		\end{enumerate}
	\end{assertion}
	\begin{proof}
		Suppose $u-v_{(1)}$ is unimodal, the first assertion is clear because $v_{(n)} - v_{(n+1)} = \mathcal{L}^{n}(u-v_{(1)})$ and $ \mathcal{L}^{n}$ is UP by assertion \rm(ii) of Lemma \ref{lem.mpp.properties}.  
		For part \rm(ii), $\mathcal{L}$ is  UR. Therefore, $(-\mathcal{L})^{n}$ is UP by properties \rm(i) and \rm(iii) of the same lemma. Thus, $(-1)^{n}\mathcal{L}^{n} \left(u-v_{(1)}\right) = (-1)^n \left\{ v_{(n)} - v_{(n+1)}\right\}$ is unimodal. 
		
		Now, suppose $K > 0$, $u \ge 0$ and $u/v_{(1)}$  is unimodal. $u/v_{(1)}$ is well defined because  $v_{(1)} \ne 0$ ($u$ is not identically zero and $K$ is positive). For part \rm(iii), given these assumptions and the fact that $\mathcal{L}^n$ is $\TP{1}$, $v_{(n)}$ and $v_{(n+1)}$ are positive.  Define $h_n \coloneqq v_{(n)} / v_{(n+1)}$. For  any $\lambda \in \mr$, 	we have, by linearity, the identity:
		\begin{alignat}{2}
			h_n -\lambda &= \frac{v_{(n)}}{v_{(n+1)}}  - \lambda = \frac{\mathcal{L}\left\{ v_{(n-1)} - \lambda v_{(n)} \right\}}{v_{(n+1)}}\nonumberj
			&= \frac{\mathcal{L}\left\{ v_{(n)}\left( v_{(n-1)}/v_{(n)} - \lambda  \right)\right\}}{v_{(n+1)}}.\label{identity.sgn.unimodal}
		\end{alignat}	
		From this identity, $\sgn(h_n-\lambda) = \sgn(v_{(n-1)}/v_{(n)}-\lambda) = \sgn(h_{n-1}-\lambda)$. Proceeding iteratively, we see at once that $\sgn(h_n-\lambda) =  \sgn(h_{n-1}-\lambda) = \cdots= \sgn(h_0-\lambda)$. Since $h_0=u/v_{(1)}$ is unimodal by assumption, $h_n$ is unimodal for all $n \ge 1$. 
		
		The final assertion  \rm(iv) is similar to part \rm(iii), but with $\mathcal{L}$ being UR. Following the same steps, we find  that for any $\lambda \in \mr$, 
		$(-1)^{n}\sgn(h_n-\lambda) =  (-1)^{n-1}\sgn(h_{n-1}-\lambda) = \cdots= \sgn(h_0-\lambda)$. Therefore, $(-1)^nh_n$ is unimodal, as required.
	\end{proof}
	We can say more if the linear transformation is invertible, as the next proposition demonstrates.
	\begin{assertion}\label{prop.recursion.dn.diff.unimodal}
		Suppose  $\mathcal{L}$ is an invertible linear transformation, and let $\mathcal{L}^{-1}$ denote its inverse transformation. Given the sequence $v_{(n)} \coloneqq \mathcal{L}^n u$,  where $u$ is an arbitrary real sequence and  $n \in \mn$, the following hold:
		\begin{enumerate}[label=\rm(\roman*)]
			\item If $\mathcal{L}^{-1}$ is UP and  $v_{(n)} - v_{(n+1)}$ is unimodal, then  $v_{(m-1)} - v_{(m)}$   is unimodal  for all $m \in [n]$.
			\item If $\mathcal{L}^{-1}$ is UR and $(-1)^n \left\{ v_{(n)} - v_{(n+1)}\right\}$ is unimodal, then $(-1)^{m-1} \left\{ v_{(m-1)} - v_{(m)}\right\}$   is unimodal  for all $m \in [n]$.
			
		\end{enumerate}
		Furthermore, if $K$ is positive and $u$ nonnegative, we have:	
		\begin{enumerate}[label=\rm(\roman*), start=3]
			\item If $\mathcal{L}^{-1}$ is UP and $v_{(n)} / v_{(n+1)}$  is unimodal, then $v_{(m-1)}/v_{(m)}$is unimodal  for all $m \in [n]$.		
			\item If $\mathcal{L}^{-1}$ is UR and  $(-1)^n v_{(n)} / v_{(n+1)}$ is unimodal, then  $v_{(m-1)}/v_{(m)}$  is unimodal  for all $m \in [n]$.
		\end{enumerate}	
	\end{assertion}
	\begin{proof}
		Part \rm(i) is obvious, since  $v_{(n-1)} - v_{(n)}= \mathcal{L}^{-1}(v_{(n)} - v_{(n+1)})$, and  $\mathcal{L}^{-m}$ is UP for every $m \in [n]$ by assertion \rm(ii) of Lemma \ref{lem.mpp.properties}.  For part \rm(iii), suppose $K > 0$ and $u \ge 0$. Given these  assumptions and the fact that $\mathcal{L}^n$ is $\TP{1}$, $v_{(n)}$ and $v_{(n+1)}$ are positive.  As before, define $h_n \coloneqq v_{(n)} / v_{(n+1)}$. For  any $\lambda \in \mr$, 	we have by the linearity of $\mathcal{L}$:
		\begin{alignat}{2}
			h_{n-1}-\lambda  
			&= \frac{\mathcal{L}^{-1}\left\{ v_{(n+1)}\left( v_{(n)}/v_{(n+1)} - \lambda  \right)\right\}}{v_{(n)}}.\label{identity.sgn.unimodal2}	
		\end{alignat}	
		Thus, $\sgn(h_{n-1}-\lambda) = \sgn(v_{(n)}/v_{(n+1)}-\lambda) = \sgn(h_{n}-\lambda)$. Iterating, we see that $\sgn(h_n-\lambda) =  \sgn(h_{n-1}-\lambda) = \cdots= \sgn(h_0-\lambda)$. Since $h_n$ is unimodal by assumption, $h_m$ is unimodal for all $m \le n$. The other two assertions are proven similarly to parts \rm(ii)-\rm(iv) of the previous proposition, with minor adjustments. 			
	\end{proof}
	
	In the final part of this section, we introduce a specific type of unimodal sequences. While the set of unimodal sequences is not closed under addition,  the sum of two sequences that are monotone in the same direction remains monotonic.  Therefore, the sum of two unimodal sequences may still be unimodal, provided certain conditions are met regarding the location of their modes.
	
	\begin{thm}\label{thm.closure.sum}
		Suppose $u$ and  $v$ are two unimodal real sequences with mode intervals $[a,b]$ and $[c,d]$, respectively, where $a,b,c,d \in \mn$, $a \le b$ and $c \le d$. If $a \le c \le b$  or $c = b+1$, then  $u+v$ is unimodal.
	\end{thm}
	\begin{proof}
		Without loss of generality, assume  the index set for both $u$ and $v$ is the set of positive integers. By definition of unimodality and the theorem's hypothesis, $u$ is nondecreasing on $[1, b]$ and nonincreasing on $[b+1, \infty)$. Similarly, $v$ is nondecreasing on $[1, c]$ and nonincreasing on $[d+1,\infty)$. Then $u+v$ is nondecreasing on $[1,\min(b,c)]=[1,c]$ and nonincreasing on $[d, \infty)$. For $u+v$ to be unimodal, it must be monotone on  $[c, d]$. If  $[a,b]$ and  $[c,d]$, coincide, then the proof is complete as $u$ and $v$ are constant on $[c, d] = [a, b]$. Since $v$ is constant on $[c, d]$,   it follows that $u+v$  is constant on $[c, \min(b,d)]$. Because it inherits the monotonicity of $u$, it is  nonincreasing in $[\min(b,d), d]$. On the final segment, $[d, \infty)$, both $u$ and $v$ are nonincreasing. In summary, $u+v$ is nondecreasing on $[a,c]$, has a mode interval $[c, \min(b,d)]$, and is nonincreasing on $[\min(b,d)+1, \infty)$. Thus, $u+v$ is unimodal. 
		
		For part \rm(ii),  given the  two mode intervals, $u$ and $v$ are nondecreasing on $[a,b]$ and $[a,d]$, respectively. Hence,  $u+v$ is nonincreasing on $[a,b]$. Also, $u+v$ is nonincreasing on $[d, \infty)$ as both $u$ and $v$ are nonincreasing. Since $v$ is constant on $[c, d]$, $u+v$ inherits the monotonicity pattern of $u$. Thus, it is nonincreasing on $[c,d]$. Depending on the values of  $u_b + v_b$ and $u_c + v_c$, $u+v$ may have a single mode at $b$ or $c=b+1$  or a double mode at $b$ and $c$. This completes the proof.
	\end{proof}

	\begin{thm}\label{cor.closure.sum}
		Consider the transformations 
		\begin{alignat}{2}
			\mathcal{L}_{m,n} &\coloneqq \mathcal{L}^m+ \cdots + \mathcal{L}^{m+n},\label{def.trans.sum}\\
			\mathcal{E}_{m, n} &\coloneqq \mathcal{J}^{2m}+ \mathcal{J}^{2m+2}+\cdots + \mathcal{J}^{2m+2n},\label{def.trans.sum.even}\\
			\mathcal{O}_{m, n}&\coloneqq \mathcal{J}^{2m+1}+ \mathcal{J}^{2m+3}+\cdots + \mathcal{J}^{2m+2n+1},\label{def.trans.sum.odd}
		\end{alignat} 
		where  $\mathcal{L}$ and $\mathcal{J}$ are  UP and UR transformations respectively and $m,n \in \mn$. For any unimodal real sequence $u$ that is either 1)  monotone or 2)  nonmonotone and its (first) mode coincides with the second-to-last term, the sequences $\mathcal{L}_{m,n}(u)$, $\mathcal{E}_{m,n}(u)$  and $-\mathcal{O}_{m,n}(u)$ are unimodal.
	\end{thm}
	\begin{proof}
		Each $\mathcal{L}^m, \ldots, \mathcal{L}^{m+n}$ is a UP transformation due to the closure property of UP transformations under composition in view of Lemma \ref {lem.mpp.properties}. If  $u$  is monotone, the result is trivial as each UP transformation produces a monotone sequence in the same direction, and their sum  is also monotone. For nonmonotone sequence $u$, we proceed inductively, using assertion \rm(ii) of Theorem \ref{thm.closure.sum} and Lemma \ref{lem.modelocation} as follows: If $u$, defined on $I=[a,b]$ where $a,b \in \mn$, has its mode at $b-1$ (and possibly $b$), then $v_k =\mathcal{L}^{k}(u)$, where $k  \in \{m, m+1, \ldots, m+n\}$, is unimodal. By Lemma \ref{lem.modelocation}, the mode of each $v_k$  is at index $b-1$ or $b$. Thus, $\mathcal{L}_{m,n}(u)=\sum_{k=m}^{m+n} v_k$  has its mode at  $b-1$ or $b$, and is therefore unimodal. 
		
		Turning to  transformation $\mathcal{E}_{m, n} $, recall that every UR transformation $\mathcal{J}$ raised to an even power is UP by  Lemma \ref{lem.mpp.properties}. The assertion follows directly from the previous result for transformation $\mathcal{L}_{m,n}$.

		Finally, consider $J = -\mathcal{O}$. In view of properties \rm(ii)-\rm(iii) in Lemma \ref{lem.mpp.properties}, raising  a UR transformation $\mathcal{J}$ to an odd power yields a UR transformation. By property \rm(i) of the same lemma, each of the transformations $-\mathcal{J}^{2m+2k+1}$, where $k=0,1,\ldots,n$, is UP.  Applying the result for transformation  $\mathcal{L}_{m,n}$ once more, we deduce that $J(u) = -\mathcal{O}(u)$ is unimodal.
	\end{proof}
	For a given sequence $u$, define the following transformations
	\begin{itemize}
		\item The zero transformation, denoted $O$, where $O(u) = (0,0,\ldots)$,
		\item The identity transformation, denoted  $I$, where $I(u) = u$.
	\end{itemize}

	\begin{lemma}\label{prop.geometric.transformations}
		Let $\mathcal{L}$ be a UP transformation and $\mathcal{J}$  be a UP or UR transformation. Suppose  both transformations are linear with a spectral radius less than 1.  For any unimodal sequence $u$ satisfying conditions 1 or 2 of Theorem \ref{cor.closure.sum}, the sequences
			$$\displaystyle \left(\frac{1}{I-\mathcal{L}}\right)(u), ~ \displaystyle\left(\frac{1}{I+\mathcal{J}}\right)(u) \text {  and } \displaystyle\left(\frac{1}{I-\mathcal{J}^2}\right)(u),$$
			are unimodal.
		\end{lemma}
		\begin{proof}
			We only prove the unimodality of the last sequence; the proofs for the other two are identical. From Theorem \ref{cor.closure.sum},  the sequence $\mathcal{E}_{0,n}(u)$, defined by (\ref{def.trans.sum.even}) with $m=0$, is unimodal for every $n=1,2,\ldots$. Since $\mathcal{J}$ is  a linear UP or UR transformation which has a spectral radius less than 1,  $\lim_{n \to \infty} \mathcal{J}^n = O$, and the series $\sum_{k \ge 1}  \mathcal{J}^k= \lim_{n \to \infty}\mathcal{E}_{0, n}$ converges  pointwise to $(1-\mathcal{J}^{2})^{-1}$. Therefore, 
			\begin{alignat*}{2}
				\frac{1}{I- \mathcal{J}^{2}}(u) = u +  \mathcal{J}^2(u) +   \mathcal{J}^4(u) + \cdots
			\end{alignat*}
			is also unimodal. This can be established by induction  using Theorem \ref{thm.closure.sum}.
		\end{proof}
		In the following section, we investigate sign-regular preservers and sign-regular reversers of  $m$-modality (resp. $m$-convexity).
		\section{Preservers and Reversers of $m$-modality/$m$-convexity}\label{sec4}
		
		Throughout this section, $I=J \coloneqq \mn \cup \{0\}$ represent the index sets of the sequences considered. Let $m > 1$, we introduce a second linear transformation $\eL_m$, which
		\begin{itemize}
			\item maps a real sequence to a real sequence as in (\ref{def.kernel.sequence-sequence}).
			\item maps a polynomial of exactly degree $k \in [m-1]$ to another polynomial of degree $k$. 
		\end{itemize}
		That is, given a real sequence $u \coloneqq (u_k)_{k \in I}$ and a  an arbitrary polynomial  of degree $k=1, \ldots, m-1$, 
		\begin{equation*}
			P_k \coloneqq P_k(x) = a_0x^k+a_1x^{k-1}+\cdots+a_{k}, ~a_0 \ne 0
		\end{equation*}
		we have:
		\begin{alignat}{2}\label{def.Mtransform}
			v_n &\coloneqq \eL_m u = \sum_ {j \ge 0} K(n, j) u_j,   &\qquad n \in J\\
			Q_k(n)&\coloneqq \eL_m P_k= \sum_ { n \ge 0} K(x, n)P_k(n) &\nonumberj
			&= b_0n^k+b_1n^{k-1}+\cdots+b_{k},  &\qquad ~b_0 \ne 0 \text{ and } n \in J,
		\end{alignat}
		where $v$ is assumed to converge uniformly on $J$.
		
		\subsection{The relation between higher-order convexity and $m$-modality}\label{sec4.1}
		
		As Karlin observed, sign-regular kernels do not automatically preserve $m$-convexity and $m$-concavity.  This section complements his work  (see Chapter 6 or p.24 in \cite{KarlinTPBook} and Section 7 in \cite{KarlinProschan})  by investigating $m$-convexity reversing and preserving transformations using kernels that are not totally positive. The following theorem will address this more precisely.
		\begin{thm}\label{thm.mconvexity.preserver}
			Let $\eL_{m}$ be the transformation given by (\ref{def.Mtransform}).  If $K$ is $\TP{m+1}$ or $-K$ is $\TN{m+1}$,  then $\eL_m$ preserves  $m$-convexity (resp. $m$-concavity). Conversely, if $-K$ is $\TP{+1}$ or  $K$ is $\TN{m+1}$, then $\eL_m$ reverses $m$-convexity (resp. $m$-concavity).
		\end{thm}
		\begin{proof}	
			The proof for the totally positive case of the first part is due to Karlin (see  p.344 in \cite{KarlinConvexity} or p.732 \cite{KarlinProschan}) and uses the identity :
			\begin{equation}\label{eqn.polynomial}
				v_n-Q_{m-1}(n) = \sum_{j \ge 0} K(n, j) \left\{u_j - P_{m-1}(j)\right\},
			\end{equation}
			where  $P_{m-1} \coloneqq  P_{m-1}(x)= a_0x^{m-1}+a_1x^{m-2}+\cdots+a_{m-2},$ and $a_0 > 0$.  Since $P_{m-1}$ is a polynomial of degree $m-1$ with a positive leading coefficient, and   
			$u$ is $m$-convex, the sequence $u-P_{m-1} \coloneqq (u_1 - P_{m-1}(1), u_2 - P_{m-1}(2),\ldots)$, changes its sign at most $m$ times. Therefore,  $v-Q_{m-1}$  has at most $m$ sign changes. By the VD property of $K$, which is $\TP{m+1}$, if $v-Q_{m-1}$ exhibits $m$ sign changes, they occur in the same order as the sign changes in $u-P_{m-1}$, namely $+-+\cdots$. Thus, $v$ is convex of order $m$, as $Q_{m-1}(n)$ is an arbitrary polynomial of degree $m-1$. When $-K$ is $\TN{m+1}$, the VD property once more ensures that $v-Q_{m-1}$ has at most $m$ sign changes. The signs of  minors  of $K$ satisfy $\varepsilon_{i}\varepsilon_{i+1}=-1$ for all $i \in [m]$. Hence, $Q_{k-1}-v$ has the opposite sign pattern to $u-P_{k-1}$ for all $k \in [m]$. This, in turn, implies that the sign patterns of $v-Q_{m-1}$ and $u-P_{m-1}$ are identical. Therefore, $v$ is $m$-convex.

			For the second part, where $ -K$ is $\TP{m}$, if $u$ is $m$-convex, then $-\eL_m(u)$ is $m$-convex  by the first part of this theorem. Its additive inverse, $\eL_m(u)$ is naturally $m$-concave. Thus, $\eL_m$ is $m$-convexity reversing.
		\end{proof}
		
		\begin{rem}
			$ $
			\begin{itemize}
				
				\item To ensure that $a_0b_0 > 0$, where $a_0$ and $b_0$ are the leading coefficients of polynomials $P_{m-1}$ and $Q_{m-1}$ respectively,  it is sufficient (though not necessary) to require that $\varepsilon_1 = 1$, hence the need for $K$ to be nonnegative or $\TP{1}$. 
				
				\item In the proof, the case where $g(x)-Q_{m-1}(x)$ displays less than $m$ sign changes is immaterial. The sole requirement is to ensure that at most $m$ sign changes occur, and when these changes happen, they follow the pattern $+-+\cdots$ for the convex case and the reverse pattern otherwise. 
			\end{itemize}	
		\end{rem}
		
		A generalisation of Theorem \ref{thm.unimodality.preserving} to $m$-modal sequences  becomes clear by relying on the preceding proof and Definition \ref{def.mmodal}. 
		\begin{thm}\label{thm.multimodality.preserver}
			Let $\eL_{m}$ be the transformation defined by (\ref{def.Mtransform}).  If $K$ is $\TP{2m+1}$ or $-K$ is $\TN{2m+1}$,  then $\eL_m$ preserves  $m$-modality. Conversely, if $-K$ is $\TP{2m+1}$ or  $K$ is $\TN{2m+1}$, then $\eL_m$ reverses $m$-modality.
		\end{thm}
		\begin{proof}	
			Construct, one at a time,  identities similar to equation (\ref{eqn.polynomial}) using polynomials of degree $2m-1$, $2m-2$ and $2m-3$, respectively. Then, adapt the proof of Theorem \ref{thm.mconvexity.preserver}, carefully tracking the sign changes in each case.			 	
		\end{proof}

		Definitions \ref{def.convex} and \ref{def.mmodal} suggest the following relationship between $m$-convex, $m$-concave sequences and multimodal sequences.
		\begin{thm}
			Let $m \ge 1$ and suppose $u$ is a regular sequence.  $u$ is $m$-modal iff $u$ is either $2m$-concave, $(2m-1)$-concave,  $(2m-1)$-convex or $(2m-2)$-convex.
		\end{thm}
		\begin{proof}
			Since  $u$ is regular, Theorem \ref{thm.multimodal.maximum} implies that it is unimodal or multimodal.
			For the sufficiency part, the proof follows by a  sign analysis. If $u$ is $2m$-concave,  there exists a polynomial $p_1$ of degree $2m-1$ such that $u-p_1$ has $2m$ sign changes, with the pattern $-+-\cdots-$. In this case, there are $m$ plus signs, so $u$ is $m$-modal. If $u$ is  $(2m-1)$-convex or $(2m-1)$-concave,  there exists a polynomial $p_2$ of order $2m-2$ such that $u-p_2$ has $2m-1$ sign changes, with the pattern $+-+\cdots-$ in the convex instance or $-+-\cdots+$ in the concave instance. In both cases, there are $m$ plus signs, so  $u$ is $m$-modal. The last case, using a polynomial of degree $2m-3$, is proven similarly.
			
			For the necessity part, suppose $u$ is $m$-modal. By Definition \ref{def.mmodal},   for any  polynomial $p_1$ of degree $2m - 1$, $u-p_1$  has at most $2m$ sign change,  $m$ of which are plus signs. With $2m$ sign changes, the pattern is $-+-\cdot-$, so $u$ is $2m$-concave. If no such $p_1$ exists, consider polynomials $p_2$   of degree $2m - 2$, for which $u-p_2$ has at most $2m-1$ sign change with $m$ plus sign. When this happens, the pattern observed is either $+-+\cdots-$ or $-+\cdots+$,  corresponding to $(2m-1)$-convex and $(2m-2)$-convex sequences, respectively. In the last scenario, where no such $p_1$ nor $p_2$ exists, consider a polynomial $p_3$ or order $2m-3$. In this case, $u-p_3$ has $2m-2$ sign changes,  $m$ of which are plus signs, with the pattern $+-+\cdots+$. This corresponds to an $(2m-2)$-convex sequence. This completes the proof.
		\end{proof}

		\subsection{$m$-modality and $m$-convexity of a quotient of sequences}
		Theorems 2.5 and 2.6 in \cite{Karp} incorrectly asserted that sign-regularity of kernels, regardless of the the signs of their minors $\varepsilon_1, \varepsilon_2, \ldots$, guarantee the  unimodality of quotients of series and of integral transforms. The error stems from the implicit   assumption that reversing the sign pattern of a unimodal function, unconditionally yields a unimodal function. This has been addressed in Lemmas (\ref{lem.sign.chg})-(\ref{lem.bimodal}). We provide a correction to their claim (assertions \rm(iii) and \rm(iv) in the next theorem) and generalise their results to  $m$-modal and $m$-convex sequences.
		\begin{lemma}\label{lemma.quotient.modal}
			Let $u, v$ be two real sequences on $I$. Suppose that $u$, $v$, and their additive inverses are regular sequences. If $v \ne 0$, then $w=u/v$ is regular.
		\end{lemma}
		\begin{proof}
			This is evident from the definition of regular sequences as all values of $u$ and $v$ are attained.
		\end{proof}
		\begin{thm}\label{thm.multimodal.quotient}
			Fix $m \ge 1$ and let $\eL$ be the transformation defined in (\ref{def.kernel.sequence-sequence}), based on a $\SR{2m+1}$ kernel $K: I \times J \to \mr$.  Let  $u$ and $v$ be two regular sequences on $I$, with regular additive inverses. Assume $\sum_{i \ge 0}  K(n, i) u_i < \infty$, $\sum_{i \ge 0}  K(n, i) v_i < \infty$,  $v_i > 0$ and $K(\cdot, i)$ attains all of its values,  for all $i \in I$.  Define
			\begin{equation}\label{def.ratio}
				w \coloneqq  \frac{\eL u}{\eL v}. 
			\end{equation}
			If the sequence of quotients $u/v = (u_0/v_0, u_1/v_2, \ldots, )$ is $m$-modal then:
			\begin{enumerate}[label=\rm(\roman*)]
				\item There exists a $p \in [m]$ such that $w$ is a $p$-modal   when $K$ or $-K$ is $\TP{2m+1}$.
				\item There exists a $p \in [m+1]$ such that $w$ is $p$-modal  when $K$ of $-K$ is $\TN{m+1}$.
			\end{enumerate}
			In particular, when $m=1$
			\begin{enumerate}[label=\rm(\roman*), start=3]
				\item $w$ is unimodal when  $K$ or $-K$ is $\TP{3}$.
				\item $w$ is  either unimodal or bimodal when  $K$ or $-K$ is $\TN{3}$.
			\end{enumerate}
		\end{thm}
		\begin{proof}
			Since  $u, v, -u, -v$ are regular,  Lemma \ref{lemma.quotient.modal} together with Theorem \ref{thm.multimodal.maximum2} imply that the sequence of quotients $w=u/v$ is unimodal or multimodal. By assumption, the order of modality is $m$. For part \rm(i), assume without loss of generality that $K$ is $\TP{m+1}$. If $-K$ is $\TP{m+1}$, simply factor out -1 from the numerator and denominator in (\ref{def.ratio}).   For arbitrary $\lambda \in \mr$, 	we have
			\begin{alignat*}{2}
				w-\lambda &= \frac{u^\prime}{v^\prime}  - \lambda 
				= \frac{\eL\left\{ v\left( w- \lambda  \right)\right\}}{v^\prime},
			\end{alignat*}	
			where $u^\prime \coloneqq \eL u $ and $v^\prime \coloneqq \eL v$.	Since $K$ is $\TP{m+1}$, it is nonnegative as it is  $\TP{1}$.  Thus, $v^\prime > 0$. In this case, $\sgn(v(w-\lambda)) = \sgn(w-\lambda)$ because $v > 0$.  Since $w$ is $m$-modal, $w -\lambda$ can have up to $2m$ sign changes.  By the VD property of $K$, $S(w) \le 2m$. If $S(w)=S(u/v) = k\le 2m$, then  $w$ and $u/v$  have the same sign pattern in view of part \rm(ii) of Theorem \ref{def.vd}. When $k=2m$ or $2m-1$,  $u/v$ is $m$-modal and so is $w$. Proceeding iteratively for each value of $k$, we conclude that $w$ is $p$-modal with $p \in [m]$, which is the desired outcome. If $S(w)< S(u/v)$, the sign pattern of $w$ is immaterial, as it will be $p$-modal, where $p \le m$.

			For part \rm(ii),  assume without loss of generality that $-K$ is $\TN{2m+1}$. In this case, $K$ is nonegative and has alternating minor signs  $\varepsilon_{1}=1, \varepsilon_{i}\varepsilon_{i+1}=-1$ for all $i$.  This implies $v^\prime > 0$. Using the  VD property, once more, together with part \rm(ii) of Theorem \ref{def.vd}, we get $S(w) \le 2m$. If  $S(w)=S(u/v) = k\le 2m$, then $u/v$ and  $w$ have the opposite sign patterns. Specifically, when $k=2m$, $w$ is $m$-modal, but because $w$ has the opposite sign pattern, it is an $(m+1)$-modal  as demonstrated in Theorem \ref{prop.reversal.multimodal}. As in part \rm(i), we conclude that $w$ is $p$-modal with $p \in [m+1]$. Part \rm(iii) and \rm(iv) follow directly from \rm(i) and \rm(ii). This concludes the proof.
		\end{proof}
		
		Adapting this theorem to convex (or concave) quotients yields a slightly stronger result. This is because we do need to assume $w$ is regular.  Furthermore, an $\SR{2m+1}$  kernel with the appropriate signature may transform an $m$-modal sequence into an $(m-1)$-modal sequence.   In the $m$-convexity case, we only require convergence and boundedness of the kernel.  Thus, the requirement for $m$-convexity preservation/reversion is less restrictive.
		
		\begin{thm}\label{thm.mconvexity}
			Fix $m \ge 1$ and let $\eL_m$ be the transformation defined in (\ref{def.Mtransform}) based on a $\SR{2m+1}$ kernel $K: I \times J \to \mr$.  Let  $u$ and $v$ be two bounded real sequences on $I$, where  $\sum_{i \ge 0}  K(n, i) u_i < \infty$, $\sum_{i \ge 0}  K(n, i) v_i < \infty$, and $v_i > 0$ for all $i \in I$. Define
			\begin{equation*}
				w \coloneqq  \frac{\eL_m u}{\eL_m v}. 
			\end{equation*}
			If the sequence of quotients $u/v = (u_0/v_0, u_1/v_2, \ldots, )$ is  $m$-convex (resp. $m$-concave), then:
			\begin{enumerate}[label=\rm(\roman*)]
				\item  $w$ is $m$-convex (resp. $m$-concave) on $J$, when $K$ or $-K$ is $\TP{m+1}$  then
				\item 	$w$ is $m$-concave (resp. $m$-convex) on $J$, when $K$ of $-K$ is $\TN{m+1}$.
			\end{enumerate}
			In particular, when $m=1$
			\begin{enumerate}[label=\rm(\roman*), start=3]
				\item $w$ is convex (resp. concave) on $J$ when  $K$ or $-K$ is $\TP{3}$.
				\item $w$ is  $m$-concave (resp. $m$-convex) on $J$  when  $K$ or $-K$ is $\TN{3}$.
			\end{enumerate}
		\end{thm}
		\begin{proof}
			Let $P_{m-1} \coloneqq  P_{m-1}(x)= a_0x^{m-1}+a_1x^{m-2}+\cdots+a_{m-2},$ where $a_0 > 0$.  The proof uses the identity
			$$v_n-Q_{m-1}(n) = \frac{\sum_{j \ge 0} K(n, j) \left\{v_j(w_j - P_{m-1}(j))\right\}}{v^\prime_n},~n\in J$$
			where $Q_{m-1} \coloneqq \eL_m P_{m-1} $. Given that $K$ is $\TP{m+1}$, it follows that $g_v(x) > 0$. By the theorem's hypothesis $v > 0$,  this imples that $\sgn(v_n(w_n - P_{m-1}(n))) = \sgn(w_n - P_{m-1}(n))$. By the VD property of $K$, the inequality $S(g(x)-Q_{m-1}(x)) \le m$ holds. In the event where $S(g(x)-Q_{m-1}(x))=m$,  given that $\varepsilon_{m}\varepsilon_{m+1}=1$, part \rm(ii) of Theorem \ref{def.vd} dictates that the sequence $w - P_{m-1}(n)$ and the function $g(x)-Q_{m-1}(x)$ display the same pattern as $n$ traverses $I$ and $x$ traverses $J$. Consequently, $g(x)$ is $m$-convex as it has at most $m$ sign changes, and when all these changes  occur, the sign arrangement matches that of $w-P_{m-1}$, namely $+-+\cdots$. The case of $m$-concavity is analogous.  

			Part \rm(ii) is straightforward and can be proved in a similar way to part  \rm(ii) of the previous theorem.  Parts \rm(iii) and \rm(iv) follow from \rm(i) and \rm(ii). This completes the proof.
		\end{proof}

		\subsection{Applications and Examples}
		\subsubsection{Domination Polynomials}
		A dominating set $S$ of a graph $G = (V, E)$ is a subset of vertices $S$ such that every vertex in $G$ is either in $S$ or adjacent to a vertex in $S$. The domination polynomial of $G$ is a polynomial whose degree is the number of vertices in $G$ and whose $n$th coefficient $n$ is the number of dominating sets of  $G$ of cardinality $n$. \cite{Beaton}  extended the families for which unimodality of the domination polynomial is known to other types of graphs.  We offer a generalisation to Theorem 2.2 in their result and offer a simpler proof.
		\begin{thm}
			Fix two integers $m \ge 1$ and $k \ge 0$ and consider the sequence of polynomials, $(f_{n})_{n\ge1}$ defined by the recurrence
			\begin{equation}\label{def.sum}
				f_{n} (x)\coloneqq\sum_{i=n-m}^{n-1} x^k f_{i}(x),.
			\end{equation}
			If $f_{m+1}$ is unimodal (resp. convex, convave) then $f_{n}$ is unimodal (resp. convex, concave) for all $n > m$.
		\end{thm}
		\begin{proof}
			First, we may express the first $m$ polynomials of arbitrary orders $r_1 \le r_2 \le \cdots \le r_m$, respectively, using formal power series $f_i(x) = \sum_{j = 0}^\infty a_{i,j} x^j$, where $i \in [m]$ and $a_{i,j}=0$ when $j > r_i$. Each of these polynomials can be represented by the sequences $(a_{i,0}, a_{i, 1}, \ldots, a_{i,r_i}, 0,0, \ldots)$.  For $n > m$, we have
			$$f_{n} (x) \coloneqq \sum_{i=n-m}^{n-1} x f_i(x) = a_{n,\,r_m+k} x^{r_m+k}+a_{n,\,r_m+k-1}  x^{r_m+k-1}+\cdots+a_{n,\,k}x^{k}$$
			where
			\begin{alignat*}{2}
				a_{n, j} \coloneqq 
				\begin{dcases}
					a_{n-1,\, j-k} + a_{n-2,\, j-k} + \cdots +a_{n-m,\, j-k},&  j \ge k,\\
					0 & otherwise
				\end{dcases}
			\end{alignat*}
			For $n > m$, $a_{n, j}$ can be expressed as
			\begin{equation}\label{def.convolution}
				a_{n, j} = \sum_{i= 0}^\infty K(n, i) a_{i, j-k},
			\end{equation}
			where $K$ is the kernel representation of the sequence $(0, \overbrace{1,1,\ldots,1}^{m \text { times}}, 0,\ldots,0)$ after setting all entries in the first $m$ rows to 0. An example of a $6 \times 6$ matrix $K$ with $m=2$ is given below
			$$K=\begin{pmatrix}
				0 & 0 & 0 & 0 & 0 & 0 \\
				0 & 0 & 0 & 0 & 0 & 0 \\
				1 &1 & 0 & 0 & 0 & 0 \\
				0 & 1& 1 & 0 & 0 & 0 \\
				0 & 0 & 1 & 1 & 0 & 0 \\
				0 & 0 & 0 & 1 & 1 & 0
			\end{pmatrix}.$$
			This kernel is trivially $\TP{}$. Since the sequence  $(a_{m+1,0}, a_{m+1,1}, \ldots, )$ is unimodal  by assumption, Theorem \ref{thm.unimodality.preserving} implies the sequence $(a_{n,0}, a_{n1}, \ldots, )$ is unimodal for all $n > m$ and hence the polynomial $f_{n}$ is unimodal for all $n > m$. The convexity and concavity results are proved in exactly the same way by using Theorem \ref{thm.mconvexity.preserver}.
		\end{proof}
		Note that Theorem 2.2 in \cite{Beaton} is obtained directly by setting $m=1$ and $k=1$ in (\ref{def.sum}).
		\subsubsection{Numerical examples}						
		We present three different scenarios of  sequence-to-sequence transformations using a totally negative matrix, which naturally maps a polynomial of degree $m$ to a polynomial of degree $m$. Let
		\[
		A = \begin{pmatrix}
			-1 & -2 & -3 & -4 \\
			-5 & -6 & -7 & -8 \\
			-9 & -10 & -11 & -11 \\
			-13 & -14 & -15 & -11
		\end{pmatrix}.
		\]
		We examine the quotient $u/v$ and the corresponding transformation $Au/Av$, where $u, v$ are two given vectors.

		\noindent 1. Unimodal to Bimodal: Let $u= (0,  3 , 3 , 1)$, $v= (1 , 1 , 1 ,1).$ 
		By direct computation, $u/v= (0,  3,  3, 1)$ is unimodal. However, 	
		$$ \frac{Au}{Av}
		= (	1.9, 1.8077, 1.8049, 1.8491),
		$$
		is  bimodal. Our result does not imply that a unimodal sequence will necessarily become turn bimodal; it only states that the transformed sequence will be $\{0,1, 2\}$-modal.
		
		\noindent 2. Bimodal to Unimodal: Let $u = (6, 5, 6, 7)$, $v = (3, 1, 2, 1).$ 
		In this example,  $u/v= (2, 5, 3, 7)$ is bimodal and 	
		$$\frac{Au}{Av}
		= (	4.1333, 3.6744, 3.5286, 3.3511),
		$$
		is unimodal.
		
		\noindent 3. Convex to Concave: Let $u = (4, 2, 1, 2)$, $v = (1, 1, 2, 2).$ 
		In this example,  $u/v= (4, 2, 1/2, 1)$ is convex. However,
		$$\frac{Au}{Av}
		= (1.1176, 1.3415, 1.4127, 1.481),
		$$
		turns out to be  concave. In fact, $A$ maps 3-convex sequences $u$ and $v$ into 3-convex sequences $Au$ and $Av$.
		\section{Conclusion}
		
		In contrast to convexity, which is well studied and has been covered in many references, unimodality, despite its simplicity, received less attention outside of probability theory \cite{Beaton, Dharmadhikar} , modelling preferences in economics, marketing science and social choice theory \cite{Vanblokland, Ansari}. Given that most applications tend to focus on continuous and bounded functions on a closed interval or on a finite sequence, where each value is necessarily attainable, questions related to the existence or conditions of unimodality of a function to hold where not required. It is quite clear from the definition of probability distributions that a unimodal (distribution) function  has an atom only at a unique point (the mode).
		
		The recent study by \cite{Karp} highlighted the need for a formal existence as well as conditions for a function to be unimodal. Our attempt goes into the same vein, with an attempt to generalise the treatment to encompass the concept of multimodality. Turning to the topic of total positivity, we  derived important results that specify the type of sign-regular kernels that preserve unimodality and convexity and those that reverse unimodality and convexity. As reversal is defined in terms of sign inversion, we have shown that  reversing a unimodal sequence yields a modal sequence of order 0, 1  or 2. Zero being nonmodal. Finally, we used the  machinery of total positivity and that of generalised unimodality to find important criteria for the $m$-modality and $m$-convexity of the quotient of transformed sequences. These results are portable  to power series, more general series, and integral transforms, as we shall show in our forthcoming work on $m$-modality of functions.

		While convexity has been extensively studied, unimodality, despite its simplicity, has received less attention outside of probability theory  \cite{Beaton, Dharmadhikar} and its applications to modelling preferences in economics, marketing, and social choice theory \cite{Vanblokland, Ansari}.  Because most applications focus on continuous, bounded functions on closed intervals or finite sequences where all values are attainable, questions of existence and conditions for unimodality have not been central. 
		
		Recent work by \cite{Karp}  highlighted the need for formal definitions of existence and conditions for unimodality.  Our work follows this direction, generalising the treatment to include multimodality.  Regarding total positivity, we have derived key results specifying which sign-regular kernels preserve and reverse unimodality and convexity.  Since reversal involves sign inversion, we have shown that reversing a unimodal sequence results in a modal sequence of order 0, 1, or 2 (where 0 indicates nonmodality). Finally, using total positivity , we have established important criteria for the $m$-modality and $m$-convexity of quotients of transformed sequences. These results are applicable to power series, more general series, and integral transforms, as we will demonstrate in future work on the $m$-modality of functions

		\printbibliography
		
	\end{document}